\documentclass[12pt,british,a4paper,reqno]{amsart}
\usepackage[T1]{fontenc}
	\usepackage[latin9]{inputenc}
	\usepackage[british]{babel}

	\usepackage{amsmath}
	\usepackage{amssymb}
	\usepackage{amsthm}
	\usepackage{braket}
	\usepackage{xcolor}
	\usepackage{enumitem}
	\usepackage{fancyhdr}
	\usepackage{geometry}
	\usepackage{graphicx}
	\usepackage[hyperfootnotes=false]{hyperref}
		\usepackage{bookmark}
	\usepackage{ifsym}
	\usepackage{indentfirst}
	\usepackage{mathrsfs}
	\usepackage{mathtools}
	\usepackage{proof}
	\usepackage{qtree}
	\usepackage{setspace}
	\usepackage{soul}
	\usepackage{tensor}
	\usepackage{tikz}
	\usepackage{tikz-cd}

	\bookmarksetup{numbered,open}

	\renewcommand{\phi}{\varphi}

	\providecommand{\corollaryname}{Corollary}
	\providecommand{\definitionname}{Definition}
	\providecommand{\examplename}{Example}
	\providecommand{\lemmaname}{Lemma}
	\providecommand{\propositionname}{Proposition}
	\providecommand{\remarkname}{Remark}
	\providecommand{\theoremname}{Theorem}
	\providecommand{\setupname}{Setup}
	\providecommand{\conjecturename}{Conjecture}
	\providecommand{\questionname}{Question}

	\theoremstyle{plain}
		\newtheorem{thm}{\protect\theoremname}[section] 
		\newtheorem{prop}[thm]{\protect\propositionname}
		\newtheorem{lem}[thm]{\protect\lemmaname}
		\newtheorem{cor}[thm]{\protect\corollaryname}

	\theoremstyle{definition}
		\newtheorem{defn}[thm]{\protect\definitionname}
		\newtheorem{example}[thm]{\protect\examplename}
		\newtheorem{setup}[thm]{\setupname}

	\theoremstyle{remark}
		\newtheorem{rem}[thm]{\protect\remarkname}
		
	\numberwithin{figure}{section}
	\numberwithin{equation}{section}

	\newenvironment{acknowledgements}{
		\begin{abstract}} {\end{abstract}}

	\usetikzlibrary{matrix,arrows,decorations.pathmorphing,positioning,decorations.pathreplacing}
	\tikzset{commutative diagrams/.cd, 
		mysymbol/.style = {start anchor=center, end anchor = center, draw = none}}

	\newcommand{\commutes}[2][\circ]{\arrow[mysymbol]{#2}[description]{#1}}
	
	\newcommand{\PB}[2][\square]{\arrow[mysymbol,pos=0.27]{#2}[description]{#1}}

	\let\amph=& 

	\newcommand{\BC}{\mathbb{C}}
	
	\newcommand{\BE}{\mathbb{E}}

	\newcommand{\BR}{\mathbb{R}}

	\newcommand{\CA}{\mathcal{A}}
	\newcommand{\CB}{\mathcal{B}}
	\newcommand{\CC}{\mathcal{C}}

	\newcommand{\CH}{\mathcal{H}}

	\newcommand{\CR}{\mathcal{R}}
	\newcommand{\CS}{\mathcal{S}}
	\newcommand{\CT}{\mathcal{T}}
	\newcommand{\CU}{\mathcal{U}}
	\newcommand{\CV}{\mathcal{V}}
	\newcommand{\CW}{\mathcal{W}}
	\newcommand{\CX}{\mathcal{X}}
	\newcommand{\CY}{\mathcal{Y}}

	\newcommand{\LF}{\mathrm{LF}}

	\newcommand{\SH}{\mathscr{H}}

		\newcommand{\Ab}{\operatorname{\mathsf{Ab}}\nolimits}
		\newcommand{\BOR}{\operatorname{\mathsf{BOR}}\nolimits}

		\newcommand{\Proj}{\operatorname{\mathsf{Proj}}\nolimits}
		\newcommand{\Inj}{\operatorname{\mathsf{Inj}}\nolimits}

		\newcommand{\op}{\mathrm{op}}

		\newcommand{\Ker}{\operatorname{Ker}\nolimits}

		\newcommand{\iso}{\cong}

		\newcommand{\Hom}{\operatorname{Hom}\nolimits}

		\newcommand{\into}{\hookrightarrow}
		\newcommand{\onto}{\rightarrow\mathrel{\mkern-14mu}\rightarrow}

		\newcommand{\Cone}{\operatorname{Cone}\nolimits}
		\newcommand{\CoCone}{\operatorname{CoCone}\nolimits}
		\newcommand{\heart}{\overline{\CH}}

		\newcommand{\fs}{\mathfrak{s}}
		
	\newcommand{\deff}{\coloneqq}
	\newcommand{\eps}{\varepsilon}

	\newcommand{\ol}[1]{\overline{#1}}

\begin{document}
\title[Integral and quasi-abelian hearts of twin cotorsion pairs]{Integral and quasi-abelian hearts of twin cotorsion pairs on extriangulated categories}
\author[Hassoun]{Souheila Hassoun}
    \address{D\'{e}partment de math\'{e}matiques \\
        Universit\'{e} de Sherbrooke \\ 
        Sherbrooke, Qu\'{e}bec, J1K 2R1 \\ 
        Canada
        }
    \email{souheila.hassoun@usherbrooke.ca}
\author[Shah]{Amit Shah}
    \address{School Mathematics, Statistics and Physics\\
        Newcastle University\\
        Newcastle upon Tyne, NE1 7RU\\
        United Kingdom
        }
    \email{amit.shah@newcastle.ac.uk}
\date{\today}
%
\begin{abstract}
It was shown recently that the heart $\heart$ of a twin cotorsion pair $((\CS,\CT),(\CU,\CV))$ on an extriangulated category is semi-abelian. 
We provide a sufficient condition for the heart to be integral and another for the heart to be quasi-abelian. 
This unifies and improves the corresponding results for exact and triangulated categories. 
Furthermore, if $\CT=\CU$, then we show that the Gabriel-Zisman localisation of $\heart$ at the class of its regular morphisms is equivalent to the heart of the single twin cotorsion pair 
$(\CS,\CT)$. 
This generalises and improves the known result for triangulated categories, thereby providing new insights in the exact setting. 
\end{abstract}
\maketitle
\section{Introduction}
\label{sec:Introduction}
\emph{Extriangulated categories} (see Definition \ref{def:extriangulated-category}) were introduced in \cite{NakaokaPalu-extriangulated-categories-hovey-twin-cotorsion-pairs-and-model-structures} as a simultaneous generalisation of exact and triangulated categories. 
Many results 
(see, for example, 
\cite{IyamaNakaokaPalu-Auslander-Reiten-theory-in-extriangulated-categories}, 
\cite{ZhouZhu-triangulated-quotient-categories-revisited}, 
\cite{HuZhangZhou-gorenstein-homological-dimensions-for-extriangulated-categories}, 
\cite{Liu-localisations-of-the-hearts-of-cotorsion-pairs}, 
\cite{LiuYNakaoka-hearts-of-twin-cotorsion-pairs-on-extriangulated-categories},
\cite{ZhaoHuang-phantom-ideals-and-cotorsion-pairs-in-extriangulated-categories}, 
\cite{Zhou-filtered-objects-in-extriangulated-categories}, 
\cite{ZhuZhuang-grothendieck-groups-in-extriangulated-categories}, 
\cite{ZhuZhuang-tilting-subcategories-in-extriangulated-categories}) 
that hold for exact and triangulated categories have been \emph{unified}, i.e.\ shown to hold for extriangulated categories in general. 
Consequently, these results also hold for a range of categories that are neither exact nor triangulated, because the class of extriangulated categories is closed under certain operations (see, for example, \cite{NakaokaPalu-extriangulated-categories-hovey-twin-cotorsion-pairs-and-model-structures}, \cite{HerschendLiuYNakaoka-n-exangulated-categories}). 
Furthermore, since exact and triangulated categories play a large role in representation theory, it is natural to ask if extriangulated categories can provide new insights in this area of mathematics. 
Indeed, there have been novel applications of extriangulated categories in this field; see, for example, 
\cite{Pressland-a-categorification-of-acyclic-principal-coefficient-cluster-algebras}, 
\cite{PadrolPaluPilauPlamondon-associahedra-for-finite-type-cluster-algebras-and-minimal-relations-between-g-vectors}. 


As well as introducing extriangulated categories, Nakaoka and Palu also defined \emph{cotorsion pairs} (see Definition \ref{def:cotorsion-pair}) on extriangulated categories in \cite{NakaokaPalu-extriangulated-categories-hovey-twin-cotorsion-pairs-and-model-structures}. 
As an analogue of torsion theories defined in \cite{Dickson-torsion-theory-for-abelian-cats}, cotorsion pairs were first introduced for the category of
abelian groups in \cite{Salce-cotorsion-theories-for-abelian-groups}. 
These have since been considered in other contexts, such as in 
module (e.g.\ \cite{EklofTrlifaj-how-to-make-ext-vanish}), 
abelian (e.g.\ \cite{Hovey-cotorsion-pairs-model-category-structures-and-representation-theory}), 
exact (e.g.\ \cite{Gillespie-model-structures-on-exact-categories}) 
and triangulated (e.g.\ \cite{Nakaoka-cotorsion-pairs-I}) categories. 
(We note here that the cotorsion pairs of a Krull-Schmidt, $\Hom$-finite triangulated $k$-category ($k$ a field) are in one-to-one correspondence with its \emph{torsion theories} in the sense of \cite{IyamaYoshino-mutation-in-tri-cats-rigid-CM-mods}; see also \cite{BeligiannisReiten-homolog-homotop-aspects}.) 
In \cite{Nakaoka-cotorsion-pairs-I} Nakaoka used cotorsion pairs on triangulated categories in order to generalise the homological constructions in 
\cite{BeilinsonBernsteinDeligne-perverse-sheaves} and 
\cite{KoenigZhu-from-tri-cats-to-abelian-cats}; see also 
\cite{BuanMarshReiten-cluster-tilted-algebras}, 
\cite{KellerReiten-ct-algebras-are-gorenstein-stably-cy}. 

In \cite{BuanMarsh-BM2}, Buan and Marsh generalised the situations of \cite{KoenigZhu-from-tri-cats-to-abelian-cats} and \cite{BuanMarshReiten-cluster-tilted-algebras} in a way not covered by the results of \cite{Nakaoka-cotorsion-pairs-I}. 
However, by using the notion of a 
\emph{twin cotorsion pair} (a pair of cotorsion pairs satisfying some condition) 
and its 
\emph{heart} (a certain subfactor category of the ambient triangulated category), 
Nakaoka was able to generalise this work of Buan and Marsh; see \cite{Nakaoka-twin-cotorsion-pairs}. 
Aside from the triangulated setting, pairs of cotorsion pairs have appeared in the literature for abelian categories 
(e.g.\ \cite{Hovey-cotorsion-pairs-and-model-categories}) 
and exact categories (e.g.\ \cite{LiuY-hearts-of-twin-cotorsion-pairs-on-exact-categories}). 

In the more general context of extriangulated categories, twin cotorsion pairs (see Definition \ref{def:twin-cotorsion-pair}) were introduced and used in \cite{NakaokaPalu-extriangulated-categories-hovey-twin-cotorsion-pairs-and-model-structures} to give a bijective correspondence between certain twin cotorsion pairs and admissible model structures. 
Subsequently, although Liu and Nakaoka mainly focussed on cotorsion pairs in \cite{LiuYNakaoka-hearts-of-twin-cotorsion-pairs-on-extriangulated-categories}, they defined and studied hearts (see Definition \ref{def:heart-and-subcategories-associated-to-twin-cotorsion-pair}) of twin cotorsion pairs on extriangulated categories. 
Moreover, several results from the exact and triangulated cases were unified; in particular, it was shown that the heart of a twin cotorsion pair is always \emph{semi-abelian} (see Definition \ref{def:semi-abelian-category} and Theorem \ref{thm:heart-is-semi-abelian}). 

In this article we study further the heart of a twin cotorsion pair on an extriangulated category, and we provide more unification of the exact and triangulated settings. 
More precisely, we give a sufficient condition for the heart of a twin cotorsion pair on an extriangulated category to be \emph{integral} and another for the heart to be \emph{quasi-abelian} (see Definitions \ref{def:integral-category} and \ref{def:quasi-abelian-category}, respectively). 
Examples of categories that are both integral and quasi-abelian include: 
any abelian category; 
the category of topological abelian groups (see \cite[\S 2]{Rump-almost-abelian-cats}); 
and the quotient category $\CC/[\CX_{R}]$, where $\CC$ 
is a cluster category (in the sense of \cite{BMRRT-cluster-combinatorics}), 
$R\in\CC$ is a rigid object and $[\CX_{R}]$ the ideal of morphisms factoring through $\CX_{R}=\Ker(\CC(R,-))$ (see \cite[Cor.\ 3.10]{BuanMarsh-BM2} and \cite[Thm.\ 5.5]{Shah-quasi-abelian-hearts-of-twin-cotorsion-pairs-on-triangulated-cats}). 
However, the classes of integral and quasi-abelian categories do not coincide; see Examples \ref{exa:Hausdorff-topological-abelian-groups} and \ref{exa:integral-not-quasi-abelian-category-example}. 

Suppose $((\CS,\CT),(\CU,\CV))$ is a twin cotorsion pair on an extriangulated category $\CB$, and let $\heart$ denote its heart. 
In Theorem \ref{thm:quasi-abelian-heart} we provide a sufficient condition for $\heart$ to be quasi-abelian. 
Suppose also that $\CB$ has enough projectives and enough injectives (see Definition \ref{def:projectives-ProjB-has-enough-projectives-duals}). 
We prove that $\heart$ is an integral category whenever $\CU\subseteq\CS * \CT$ and the subcategory of projective objects in $\CB$ is contained in $\CW = \CT\cap\CU$ (see Theorem \ref{thm:integral-heart}). 
Consequently, we show that if $\CT\subseteq\CU$ or $\CU\subseteq \CT$, then $\heart$ is integral and quasi-abelian (see Corollary \ref{cor:heart-is-integral-and-quasi-abelian-if-U-in-T-or-T-in-U}). 

Specialising further in \S\ref{sec:localisation-of-an-integral-heart}, we assume $\CT=\CU$. 
Thus, the heart $\heart$ is integral by Corollary \ref{cor:heart-is-integral-and-quasi-abelian-if-U-in-T-or-T-in-U}. 
Recall that a morphism is \emph{regular} if it is both a monomorphism and an epimorphism, and let $\CR$ denote the class of regular morphisms in $\heart$. 
Then the (Gabriel-Zisman) localisation $\heart_{\CR}$ (see \cite{GabrielZisman-calc-of-fractions}) of $\heart$ at $\CR$ is an abelian category by \cite[Thm.\ 4.8]{BuanMarsh-BM2}. 
The heart $\ol{\SH}_{(\CS,\CT)}$ of the single cotorsion pair $(\CS,\CT)$ is also an abelian category (see \cite[Thm.\ 3.2]{LiuYNakaoka-hearts-of-twin-cotorsion-pairs-on-extriangulated-categories}), and we prove that $\heart_{\CR}$ and $\ol{\SH}_{(\CS,\CT)}$ are equivalent (see Theorem \ref{thm:T-is-U-implies-localisation-of-H-at-regular-morphisms-is-equivalent-to-heart-of-cotorsion-pair-S-T}). 

Since we do not assume $\CB$ to be a Krull-Schmidt category, 
Theorem \ref{thm:quasi-abelian-heart} 
(respectively, Theorem \ref{thm:integral-heart}, Theorem \ref{thm:T-is-U-implies-localisation-of-H-at-regular-morphisms-is-equivalent-to-heart-of-cotorsion-pair-S-T}) 
gives an improvement of  
\cite[Thm.\ 7.4]{LiuY-hearts-of-twin-cotorsion-pairs-on-exact-categories} 
(respectively, 
\cite[Thm.\ 6.2]{LiuY-hearts-of-twin-cotorsion-pairs-on-exact-categories}, 
\cite[Thm.\ 4.8]{Shah-quasi-abelian-hearts-of-twin-cotorsion-pairs-on-triangulated-cats}).

This paper is organised as follows. 
In \S\ref{sec:preliminaries} we recall the necessary background material on non-abelian categories, extriangulated structures and twin cotorsion pairs. 
In \S\ref{sec:the-case-when-H-is-integral} and \S\ref{sec:the-case-when-H-is-quasi-abelian}, respectively, we give sufficient assumptions for the heart of a twin cotorsion pair to be integral and quasi-abelian, respectively. 
And in \S\ref{sec:localisation-of-an-integral-heart} we show there is an explicit equivalence from the heart of a cotorsion pair $(\CS,\CT)$ to a localisation of the heart of a twin cotorsion pair of the form $((\CS,\CT),(\CT,\CV))$. 

\section{Preliminaries}\label{sec:preliminaries}
%
%
\subsection{Preabelian categories}
\label{sub:preabelian-categories}
The main results of \S\ref{sec:the-case-when-H-is-integral} and \S\ref{sec:the-case-when-H-is-quasi-abelian} give a sufficient condition for the heart of a twin cotorsion pair on an extriangulated category to be an integral and a quasi-abelian category, respectively. 
We recall briefly these definitions in this section. 
For more details we refer the reader to \cite{Rump-almost-abelian-cats}.
\begin{defn}
\label{def:preabelian-category}
\cite[p.\ 24]{Popescu-abelian-cats-with-apps-to-rings-and-modules}, 
\cite[\S 5.4]{BucurDeleanu-intro-to-theory-of-cats-and-functors} 
A \emph{preabelian} category is an additive category in which every
morphism has a kernel and a cokernel.
\end{defn}
%
%
We will see later that the heart of any twin cotorsion pair on an extriangulated category is always semi-abelian (see Theorem \ref{thm:heart-is-semi-abelian}). 
We recall the definition of such a category now. 
Let $\CA$ be a preabelian category. 
\begin{defn}
\label{def:semi-abelian-category}\cite[p.\ 167]{Rump-almost-abelian-cats}
We call $\CA$ \emph{left semi-abelian} if each morphism $f\colon A\to B$ factorises as $f=ip$ for some monomorphism $i$ and cokernel $p$. 
Dually, we call $\CA$ \emph{right semi-abelian} if each morphism $f$ decomposes as $f=ip$ with $i$ a kernel and $p$ some epimorphism. 
And $\CA$ is called \emph{semi-abelian} when it is both left and right semi-abelian. 
\end{defn}
Now we recall the main definitions of this section. 
%
%
%
\begin{defn}
\label{def:integral-category}\cite[p.\ 168]{Rump-almost-abelian-cats}
The category $\CA$ is called \emph{left integral} if epimorphisms are stable under pullback. 
Dually, $\CA$ is called \emph{right integral} if monomorphisms are stable under pushout. 
And if $\CA$ is both left and right integral, then it is called \emph{integral}.
\end{defn}
\begin{defn}
\label{def:quasi-abelian-category}\cite[p.\ 168]{Rump-almost-abelian-cats}
The category $\CA$ is called \emph{left quasi-abelian} if cokernels are stable under pullback. 
Dually, $\CA$ is called \emph{right quasi-abelian} if kernels are stable under pushout. 
And if $\CA$ is both left and right quasi-abelian, then it is called \emph{quasi-abelian}.
\end{defn}
%
%
%
The following two results are used in the proofs of our main results in \S\S\ref{sec:the-case-when-H-is-integral}--\ref{sec:the-case-when-H-is-quasi-abelian}. 
%
%
%
\begin{prop}\label{prop:semi-abelian-category-is-left-integral-iff-right-integral}
\emph{\cite[p.\ 173, Cor.]{Rump-almost-abelian-cats}} 
A semi-abelian category is left integral if and only if it is right integral.
\end{prop}
\begin{prop}\label{prop:semi-abelian-category-is-left-quasi-abelian-iff-right-quasi-abelian}
\emph{\cite[Prop.\ 3]{Rump-almost-abelian-cats}} 
A semi-abelian category is left quasi-abelian if and only if it is right quasi-abelian.
\end{prop}
We conclude this section with some examples. 
In particular, Examples \ref{exa:Hausdorff-topological-abelian-groups} and \ref{exa:integral-not-quasi-abelian-category-example} demonstrate that 
the class of integral categories is not contained in the class of quasi-abelian categories, and 
\emph{vice versa}. 
\begin{example}\label{exa:abelian-implies-integral-and-quasi-abelian}
Every abelian category is both an integral category and a quasi-abelian category.
\end{example}
\begin{example}\label{exa:integral-or-quasi-abelian-implies-semi-abelian}
Rump showed that every integral category and every quasi-abelian category is semi-abelian; see \cite[p.\ 169, Cor.\ 1]{Rump-almost-abelian-cats}. 
\end{example}
\begin{example}\label{exa:Hausdorff-topological-abelian-groups}
The category of Hausdorff topological abelian groups is quasi-abelian, but not integral (see \cite[\S2.2]{Rump-almost-abelian-cats}). 
\end{example}
\begin{example}\label{exa:integral-not-quasi-abelian-category-example}
Let $k\in \{ \BR, \BC\}$. 
Consider the category $\BOR$ of \emph{bornological locally convex spaces} over $k$; see, for example, \cite[p.\ 50]{Hogbe-Nlend-Bornologies-and-functional-analysis}. 
Bonet and Dierolf showed that $\BOR$ is not quasi-abelian in \cite{BonetDierolf-the-pullback-for-bornological-and-ultrabornological-spaces}. 
However, it has recently been shown that $\BOR$ is integral; see \cite[Thm.\ 3.5]{HassounShahWegner-examples-and-non-examples-of-integral-categories}. 
\end{example}
%
%
%
\subsection{Extriangulated categories}
\label{sec:extriangulated-categories}
In this section, we recall the theory of extriangulated categories that we will need. 
See \cite{NakaokaPalu-extriangulated-categories-hovey-twin-cotorsion-pairs-and-model-structures} for more details.
\begin{setup}
Throughout the rest of this paper, $\CB$ denotes an additive category and we assume $\CB$ is equipped with a biadditive functor $\BE \colon \CB^{\op} \times \CB \to \Ab$, where $\Ab$ denotes the category of all abelian groups.
\end{setup}
We note that a biadditive functor was called an `additive bifunctor' in \cite[\S 1.2, p.\ 649]{DraxlerReitenSmaloSolberg-exact-categories-and-vector-space-categories}. 
Moreover, each morphism $f\colon X\to Y$ in $\CB$ gives rise to abelian group homomorphisms 
$\BE(C,f)\colon \BE(C,X)\to \BE(C,Y)$ 
and 
$\BE(f^{\op},A)\colon \BE(Y,A)\to\BE(X,A)$ 
for objects $A,C\in\CB$. 
By abuse of notation, we will write $\BE(f,A)$ instead of $\BE(f^{\op},A)$. 
\begin{defn}
\label{def:extension}
\cite[Def.\ 2.1, Rem.\ 2.2, Def.\ 2.3]{NakaokaPalu-extriangulated-categories-hovey-twin-cotorsion-pairs-and-model-structures} 
An element $\delta$ of $\BE(C,A)$ for some objects $A,C$ in $\CB$ is called an $\BE$-extension, or simply an \emph{extension}.

For morphisms $f\colon A\to X$ and $g\colon Y\to C$ and for an extension $\delta\in\BE(C,A)$, we obtain the new extensions: 
\[
f_{*}\delta:=\BE(C,f)(\delta)\in\BE(C,X)
\hspace{1cm}
\text{and} 
\hspace{1cm}
g^{*}\delta:=\BE(g,A)(\delta)\in\BE(Y,A).
\] 

Let $\delta\in \BE(C,A)$ and $\delta'\in\BE(D,B)$ be any extensions. 
A \emph{morphism of extensions} 
$\delta\to\delta'$ 
is a pair $(f,h)$ of morphisms $f\colon A\to B$ and $h\colon C\to D$ in $\CB$, such that $ f_{*}\delta  = h^{*}\delta'$.
\end{defn}
We use the following facts without reference throughout the remainder of the article. 
Let $\delta\in \BE(C,A)$ be an extension. 
Then any morphism $f\colon A\to X$ in $\CB$ induces a morphism $(f,1_{C})\colon \delta\to f_{*}\delta$ of extensions. 
Similarly, any morphism $g\colon Y\to C$ in $\CB$ gives rise to a morphism $(1_{A},g)\colon g^{*}\delta\to \delta$. 
\begin{defn}\label{def:equivalent-sequences}
\cite[Def.\ 2.7]{NakaokaPalu-extriangulated-categories-hovey-twin-cotorsion-pairs-and-model-structures} 
Let $A,C$ be objects in $\CB$. Two sequences 
$\begin{tikzcd}[column sep=0.7cm]
A \arrow{r}{a}& B\arrow{r}{b} & C 
\end{tikzcd}$ 
and 
$\begin{tikzcd}[column sep=0.7cm]
A \arrow{r}{a'}& B'\arrow{r}{b'} & C 
\end{tikzcd}$ 
of composable morphisms in $\CB$ are said to be \emph{equivalent} if there exists an isomorphism 
$g\colon B\to B'$ such that the diagram 
\[
\begin{tikzcd}
A \arrow{r}{a} \arrow[equals]{d}& B\arrow{r}{b} \arrow{d}{g}[swap]{\iso}& C \arrow[equals]{d}\\
A \arrow{r}{a'} & B'\arrow{r}{b'} & C 
\end{tikzcd}
\] 
commutes. 
This determines an equivalence relation on the class of sequences of the form 
$\begin{tikzcd}[column sep=0.7cm]
A \arrow{r}{}& -\arrow{r}{} & C.
\end{tikzcd}$ 
We denote by 
$[\begin{tikzcd}[column sep=0.7cm]
A \arrow{r}{a}& B\arrow{r}{b} & C 
\end{tikzcd}]$ 
the equivalence class of the sequence 
$\begin{tikzcd}[column sep=0.7cm]
A \arrow{r}{a}& B\arrow{r}{b} & C. 
\end{tikzcd}$
\end{defn}
\begin{setup}
Throughout the rest of this paper, let $\fs$ be a correspondence that, for each pair of objects $A,C$ in $\CB$, assigns to each extension $\delta\in\BE(C,A)$ an equivalence class 
$\fs(\delta) = 
[\begin{tikzcd}[column sep=0.7cm]
A \arrow{r}{a}& B\arrow{r}{b} & C 
\end{tikzcd}]$. 
\end{setup}
Soon we will see that, provided $\fs$ satisfies some conditions, $\fs$ will allow us to visualise the abstract structure of $(\CB,\BE)$. 
The next two definitions make this more precise.
\begin{defn}\label{def:realisation-of-E}
\cite[Def.\ 2.9]{NakaokaPalu-extriangulated-categories-hovey-twin-cotorsion-pairs-and-model-structures} 
Let $\delta\in\BE(C,A)$ and $\delta'\in\BE(C',A')$ be any pair of extensions. 
Suppose $\fs(\delta) = 
[\begin{tikzcd}[column sep=0.7cm]
A \arrow{r}{a}& B\arrow{r}{b} & C 
\end{tikzcd}]$ 
and 
$\fs(\delta') = 
[\begin{tikzcd}[column sep=0.7cm]
A' \arrow{r}{a'}& B'\arrow{r}{b'} & C' 
\end{tikzcd}]$. 
We call $\fs$ a \emph{realisation of $\BE$} if, for all morphisms $(f,h)\colon \delta \to \delta'$ of extensions, there exists $g\colon B\to B'$ such that the diagram 
\[
\begin{tikzcd}
A \arrow{r}{a} \arrow{d}{f}& B\arrow{r}{b} \arrow[dotted]{d}{\exists g} & C \arrow{d}{h}\\
A' \arrow{r}{a'} & B'\arrow{r}{b'} & C'
\end{tikzcd}
\]
commutes. 
In this case, we say $\delta$ is \emph{realised by} 
$\begin{tikzcd}[column sep=0.7cm]
A\arrow{r}{a}& B\arrow{r}{b} & C 
\end{tikzcd}$
and that $(f,h)$ is \emph{realised by} $(f,g,h)$. 
We also say 
$\begin{tikzcd}[column sep=0.7cm]
A\arrow{r}{a}& B\arrow{r}{b} & C
\end{tikzcd}$ 
\emph{realises} $\delta$, 
and $(f,g,h)$ \emph{realises} $(f,h)$. 
\end{defn}
\begin{defn}\label{def:additive-realisation}
\cite[Def.\ 2.10]{NakaokaPalu-extriangulated-categories-hovey-twin-cotorsion-pairs-and-model-structures} 
Suppose $\fs$ is a realisation of $\BE$. 
We call $\fs$ \emph{additive} if the following conditions are satisfied. 
\begin{enumerate}[label=(\roman*)]
	\item For each $A,C\in\CB$, the trivial element $0\in\BE(C,A)$ is realised by the split sequence 
		\[
		\begin{tikzcd}[ampersand replacement = \&, column sep =1.3cm]
		A \arrow{r}{\begin{psmallmatrix} 1_{A}\\0
		\end{psmallmatrix}} \& A\oplus C\arrow{r}{\begin{psmallmatrix} 0 & 1_{C}
		\end{psmallmatrix}} \& C.
		\end{tikzcd}
		\]
	\item Let $\delta\in\BE(C,A)$ and $\delta'\in\BE(C',A')$ be any extensions. 
		Let $i_{X}\colon X\into A\oplus A'$ denote the canonical inclusion for $X\in\{A,A'\}$, and 
		let $p_{Y}\colon C\oplus C' \onto Y$ denote the canonical projection for $Y\in\{C,C'\}$. 
		If $\fs(\delta) = [\begin{tikzcd}[column sep=0.7cm]A \arrow{r}{a}& B\arrow{r}{b} & C 				\end{tikzcd}]$ and
		$\fs(\delta') = [\begin{tikzcd}[column sep=0.7cm]A' \arrow{r}{a'}& B'\arrow{r}{b'} & C' 				\end{tikzcd}]$, 
		then the element 
		\[
		(i_{A})_{*}(p_{C})^{*}\delta + (i_{A'})_{*}(p_{C'})^{*}\delta' \in \BE(C\oplus C',A\oplus A')
		\]
		is realised by the direct sum 
		\[
		\begin{tikzcd}[ampersand replacement = \&, column sep=1.3cm]
		A \oplus A' \arrow{r}{\begin{psmallmatrix} a & 0 \\ 0 & a'
		\end{psmallmatrix}} \& B \oplus B' \arrow{r}{\begin{psmallmatrix} b & 0 \\ 0 & b'
		\end{psmallmatrix}} \& C \oplus C'.
		\end{tikzcd}
		\]
\end{enumerate}
\end{defn}
We are now in a position to recall the main definition of this section. 
\begin{defn}\label{def:extriangulated-category}
\cite[Def.\ 2.12]{NakaokaPalu-extriangulated-categories-hovey-twin-cotorsion-pairs-and-model-structures} 
Let $\CB$ be an additive category. A triple $(\CB, \BE, \fs)$ is called an \emph{extriangulated category} if the following axioms are satisfied. 
\begin{enumerate}[label=(ET\arabic*), wide=0pt, leftmargin=51pt, labelwidth=41pt, labelsep=10pt, align=right]
	\item\label{ET1} $\BE\colon \CB^{\op} \times \CB \to \Ab$ is a biadditive functor.
	\item\label{ET2} $\fs$ is an additive realisation of $\BE$.
	\item\label{ET3} Let $\delta\in\BE(C,A)$ and $\delta'\in\BE(C',A')$ be any extensions, and  
		suppose $\fs(\delta)=[\begin{tikzcd}[column sep=0.7cm]A \arrow{r}{a}& B\arrow{r}{b} & C 			\end{tikzcd}]$ and
		$\fs(\delta')=[\begin{tikzcd}[column sep=0.7cm]A' \arrow{r}{a'}& B'\arrow{r}{b'} & C' 				\end{tikzcd}]$. 
		For any commutative diagram 
		\[
		\begin{tikzcd}
		A \arrow{r}{a} \arrow{d}{f} & B\arrow{r}{b} \arrow{d}{g} & C \\
		A' \arrow{r}{a'}& B'\arrow{r}{b'} & C' 
		\end{tikzcd}
		\] 
		there exists $h\colon C\to C'$, such that $(f,h)\colon \delta \to \delta'$ is a morphism of 			extensions that is realised by $(f,g,h)$.
	\item[(ET$3^{\op}$)]\label{ET3op} Dual of \ref{ET3}.
	\item\label{ET4} Let $\delta\in\BE(D,A)$ and $\delta'\in\BE(F,B)$ be extensions, realised by 
		$\begin{tikzcd}[column sep=0.7cm]A \arrow{r}{a}& B\arrow{r}{a'} & D \end{tikzcd}$ and
		$\begin{tikzcd}[column sep=0.7cm]B \arrow{r}{b}& C\arrow{r}{b'} & F \end{tikzcd}$, respectively. 
		 Then there exists $E\in\CB$ and an extension $\delta''\in\BE(E,A)$ realised by 
		 $\begin{tikzcd}[column sep=0.7cm]A \arrow{r}{c}& C\arrow{r}{c'} & E \end{tikzcd}$, such that the diagram 
		 \[
		 \begin{tikzcd}
		 A \arrow{r}{a} \arrow[equals]{d}& B\arrow{r}{a'} \arrow{d}{b}& D \arrow{d}{d}\\
		 A \arrow{r}{c} 				& C\arrow{r}{c'} \arrow{d}{b'}& E \arrow{d}{e}\\
								& F \arrow[equals]{r}		& F
		 \end{tikzcd}
		 \]
		 commutes in $\CB$ and satisfies:
		 \begin{enumerate}[label=(\roman*)]
		 	\item $\fs((a')_{*}\delta') 
			= [\begin{tikzcd}[column sep=0.7cm]D \arrow{r}{d}& E\arrow{r}{e} & F \end{tikzcd}];$
			\item $d^{*}\delta'' = \delta$; and 
			\item $a_{*}\delta'' = e^{*}\delta'$.
		 \end{enumerate}
	\item[(ET$4^{\op}$)]\label{ET4op} Dual of \ref{ET4}.
\end{enumerate}
\end{defn}
\begin{example}\label{exa:examples-of-extriangulated-categories}
\cite[Exam.\ 2.13]{NakaokaPalu-extriangulated-categories-hovey-twin-cotorsion-pairs-and-model-structures} 
Examples of extriangulated categories include: 
any triangulated category; 
extension-closed subcategories of triangulated categories; 
and exact categories that are skeletally small, or have enough projectives or injectives.
\end{example}
\begin{setup}\label{set:B-is-extriangulated}
For the remainder of \S\ref{sec:preliminaries}, we suppose $(\CB,\BE,\fs)$ is an extriangulated category. 
\end{setup}
Now we recall some useful terminology that was introduced in \cite{NakaokaPalu-extriangulated-categories-hovey-twin-cotorsion-pairs-and-model-structures}. 
\begin{defn}\label{def:conflation-inflation-deflation-E-triangle}
\cite[Def.\ 2.15, Def.\ 2.19]{NakaokaPalu-extriangulated-categories-hovey-twin-cotorsion-pairs-and-model-structures} 
A sequence 
$\begin{tikzcd}[column sep=0.7cm]
A \arrow{r}{}& B\arrow{r}{} & C
\end{tikzcd}$ 
of composable morphisms in $\CB$ is called a \emph{conflation} if it realises some extension 
$\delta\in\BE(C,A)$. 
Following \cite{LiuYNakaoka-hearts-of-twin-cotorsion-pairs-on-extriangulated-categories}, in this case we write 
$\begin{tikzcd}[column sep=0.7cm]
A \arrow[tail]{r}{}& B\arrow[two heads]{r}{} & C.
\end{tikzcd}$ 

A morphism $a\in\CB(A,B)$ is called an \emph{inflation} if there exists a conflation of the form 
$\begin{tikzcd}[column sep=0.7cm]
A \arrow[tail]{r}{a}& B\arrow[two heads]{r}{} & C.
\end{tikzcd}$ 
Dually, a morphism $b\in\CB(B,C)$ is called a \emph{deflation} if there exists a conflation of the form 
$\begin{tikzcd}[column sep=0.7cm]
A \arrow[tail]{r}{}& B\arrow[two heads]{r}{b} & C.
\end{tikzcd}$ 

Suppose that we have a morphism $(f,h)\colon \delta \to \delta'$ of extensions realised by the commutative diagram 
\begin{equation}\label{eqn:morphism-of-E-triangles}
\begin{tikzcd}
A \arrow[tail]{r}{a} \arrow{d}{f}& B\arrow[two heads]{r}{b} \arrow{d}{g} & C \arrow{d}{h}\\
A' \arrow[tail]{r}{a'} & B'\arrow[two heads]{r}{b'} & C'
\end{tikzcd}
\end{equation}
in which the top row realises $\delta$ and the bottom row realises $\delta'$. 
Then we call the pair 
$(\begin{tikzcd}[column sep=0.7cm]
A \arrow[tail]{r}{a}& B\arrow[two heads]{r}{b} & C, 
\end{tikzcd} 
\delta)$ 
an \emph{$\BE$-triangle} and denote this by 
$\begin{tikzcd}[column sep=0.7cm]
A \arrow[tail]{r}{a}& B\arrow[two heads]{r}{b} & C \arrow[dashed]{r}{\delta}& {}. 
\end{tikzcd}$ 
Furthermore, we call the triple $(f,g,h)$ a \emph{morphism of $\BE$-triangles} and denote this by
\[
\begin{tikzcd}
A \arrow{r}{a} \arrow{d}{f}& B\arrow{r}{b} \arrow{d}{g} & C \arrow{d}{h} \arrow[dashed]{r}{\delta}& {}\\
A' \arrow{r}{a'} & B'\arrow{r}{b'} & C' \arrow[dashed]{r}{\delta'}& {}
\end{tikzcd}
\]
\end{defn}
We conclude this section with some results that will be used frequently later. 
For the first of these, we recall from \cite[Def.\ 3.1]{NakaokaPalu-extriangulated-categories-hovey-twin-cotorsion-pairs-and-model-structures} that by the Yoneda Lemma, each $\BE$-extension gives rise to two natural transformations as follows. 
For any $A,C$ in $\CB$ and any extension $\delta\in\BE(C,A)$, there is an induced natural transformation 
$\delta^{\sharp}\colon \CB(A,-) \Rightarrow \BE(C,-)$, 
which is given by 
$(\delta^{\sharp})_{X}(f)=f_{*}\delta$ 
for each object $X\in\CB$ and each morphism $f\colon A\to X$. 
Dually, there is also a natural transformation 
$\delta_{\sharp}\colon \CB(-,C)\Rightarrow \BE(-,A)$, 
which is given by 
$(\delta_{\sharp})_{Y}(g)=g^{*}\delta$ 
for each object $Y\in\CB$ and each morphism $g\colon Y\to C$. 
%
\begin{prop}\label{prop:exact-sequences-from-E-triangle}
\emph{\cite[Cor.\ 3.12]{NakaokaPalu-extriangulated-categories-hovey-twin-cotorsion-pairs-and-model-structures}} 
Suppose 
$\begin{tikzcd}[column sep=0.7cm]
A \arrow[tail]{r}{a}& B\arrow[two heads]{r}{b} & C\arrow[dashed]{r}{\delta} & {}
\end{tikzcd}$ 
is an $\BE$-triangle. 
Then for each $X,Y\in\CB$, the sequences
\[
\hspace*{-0.1cm}
\begin{tikzcd}[column sep=1.07cm]
\CB(C,X) \arrow{r}{-\circ b} &\CB(B,X) \arrow{r}{-\circ a} &\CB(A,X) \arrow{r}{(\delta^{\sharp})_{X}} 
& \BE(C,X) \arrow{r}{\BE(b,X)} &\BE(B,X) \arrow{r}{\BE(a,X)} &\BE(A,X)
\end{tikzcd}
\]
and 
\[ 
\hspace*{-0.1cm}
\begin{tikzcd}[column sep=1.07cm]
\CB(Y,A) \arrow{r}{a\circ -} &\CB(Y,B) \arrow{r}{b\circ -} &\CB(Y,C) \arrow{r}{(\delta_{\sharp})_{Y}} 
& \BE(Y,A) \arrow{r}{\BE(Y,a)} &\BE(Y,B) \arrow{r}{\BE(Y,b)} &\BE(Y,C)
\end{tikzcd}
\]
are exact in $\Ab$.
\end{prop}

The next two results follow from \cite[Prop.\ 1.20]{LiuYNakaoka-hearts-of-twin-cotorsion-pairs-on-extriangulated-categories} and its dual. 
See also \cite[Cor.\ 3.16]{NakaokaPalu-extriangulated-categories-hovey-twin-cotorsion-pairs-and-model-structures}. 
\begin{prop}\label{prop:variation-of-LN-prop-1-20-same-end-terms}
Let 
$\begin{tikzcd}[column sep=0.7cm]
A \arrow[tail]{r}{a}& B\arrow[two heads]{r}{b} & C\arrow[dashed]{r}{\delta} & {}
\end{tikzcd}$ 
be any $\BE$-triangle, let $f\colon A\to D$ be any morphism in $\CB$, and suppose 
$\begin{tikzcd}[column sep=0.7cm]
D \arrow[tail]{r}{d}& E\arrow[two heads]{r}{e} & C
\end{tikzcd}$ 
is a conflation realising $f_{*}\delta$. 
For any commutative diagram 
\[
\begin{tikzcd}
A \arrow[tail]{r}{a}\arrow{d}{f}& B\arrow[two heads]{r}{b}\arrow{d}{g} & C\arrow[dashed]{r}{\delta} \arrow[equals]{d}& {}\\
D \arrow[tail]{r}{d}& E\arrow[two heads]{r}{e} & C\arrow[dashed]{r}{f_{*}\delta} & {}
\end{tikzcd}
\]
there exist morphisms  $f'\colon A\to D$ and $g'\colon B\to E$, such that the conflations 
\[
\begin{tikzcd}[ampersand replacement=\&, column sep=1.5cm]
A \arrow[tail]{r}{\begin{psmallmatrix}-f \\ a\end{psmallmatrix}}\& D\oplus B\arrow[two heads]{r}{\begin{psmallmatrix}d & g'\end{psmallmatrix}} \& E
\end{tikzcd}
\hspace*{1cm}
\text{and}
\hspace{1cm}
\begin{tikzcd}[ampersand replacement=\&, column sep=1.5cm]
A \arrow[tail]{r}{\begin{psmallmatrix}-f' \\ a\end{psmallmatrix}}\& D\oplus B\arrow[two heads]{r}{\begin{psmallmatrix}d & g\end{psmallmatrix}} \& E
\end{tikzcd}
\] 
both realise $e^{*}\delta$, and $g'a=df$, $eg'=b$ and $df'=ga$.
\end{prop}
\begin{prop}\label{prop:variation-of-LN-prop-1-20-dualised-same-start-terms}
Let 
$\begin{tikzcd}[column sep=0.7cm]
A \arrow[tail]{r}{d}& D\arrow[two heads]{r}{e} & E\arrow[dashed]{r}{\delta} & {}
\end{tikzcd}$ 
be any $\BE$-triangle, let $g\colon C\to E$ be any morphism in $\CB$, and suppose 
$\begin{tikzcd}[column sep=0.7cm]
A \arrow[tail]{r}{a}& B\arrow[two heads]{r}{b} & C
\end{tikzcd}$ 
is a conflation realising $g^{*}\delta$. 
For any commutative diagram 
\[
\begin{tikzcd}
A \arrow[tail]{r}{a}\arrow[equals]{d}& B\arrow[two heads]{r}{b}\arrow{d}{f} & C\arrow[dashed]{r}{g^{*}\delta} \arrow{d}{g}& {}\\
A \arrow[tail]{r}{d}& D\arrow[two heads]{r}{e} & E\arrow[dashed]{r}{\delta} & {}
\end{tikzcd}
\]
there exist morphisms  $f'\colon B\to D$ and $g'\colon C\to E$, such that the conflations 
\[
\begin{tikzcd}[ampersand replacement=\&, column sep=1.5cm]
B \arrow[tail]{r}{\begin{psmallmatrix}f' \\ b\end{psmallmatrix}}\& D\oplus C\arrow[two heads]{r}{\begin{psmallmatrix}e & -g\end{psmallmatrix}} \& E
\end{tikzcd}
\hspace*{1cm}
\text{and}
\hspace{1cm}
\begin{tikzcd}[ampersand replacement=\&, column sep=1.5cm]
B \arrow[tail]{r}{\begin{psmallmatrix}f \\ b\end{psmallmatrix}}\& D\oplus C\arrow[two heads]{r}{\begin{psmallmatrix}e & -g'\end{psmallmatrix}} \& E
\end{tikzcd}
\] 
both realise $a_{*}\delta$, and $f'a=d$, $ef'=gb$ and $g'b=ef$.
\end{prop}
\subsection{Twin cotorsion pairs on extriangulated categories}
\label{sec:twin-cotorsion-pairs-on-extriangulated-categories}
In this section we recall the basics of the theory of twin cotorsion pairs on extriangulated categories  that we use throughout the rest of this article. 
For more details we refer the reader to \cite{NakaokaPalu-extriangulated-categories-hovey-twin-cotorsion-pairs-and-model-structures} and \cite{LiuYNakaoka-hearts-of-twin-cotorsion-pairs-on-extriangulated-categories}. 

We are still in the situation of Setup \ref{set:B-is-extriangulated}.
\begin{defn}\label{def:Cone-CoCone-subcategories}
\cite[Def.\ 4.2]{NakaokaPalu-extriangulated-categories-hovey-twin-cotorsion-pairs-and-model-structures}
Let $\CU, \CV$ be full subcategories of $\CB$ that are closed under isomorphisms. 
\begin{enumerate}[label=(\roman*)]
	\item By $\Cone(\CV,\CU)$ we denote the full subcategory of $\CB$ that consists of objects 		$X\in\CB$ for which there exists a conflation 
	$\begin{tikzcd}[column sep = 0.7cm]
	V \arrow[tail]{r}& U\arrow[two heads]{r} & X, 
	\end{tikzcd}$ 
	where $U\in\CU$ and $V\in\CV$. 
	\item By $\CoCone(\CV,\CU)$ we denote the full subcategory of $\CB$ that consists of objects 		$X\in\CB$ for which there exists a conflation 
	$\begin{tikzcd}[column sep = 0.7cm]
	X \arrow[tail]{r}& V\arrow[two heads]{r} & U, 
	\end{tikzcd}$ 
	where $U\in\CU$ and $V\in\CV$. 
\end{enumerate}
\end{defn}

\begin{defn}\label{def:cotorsion-pair}
\cite[Def.\ 4.1]{NakaokaPalu-extriangulated-categories-hovey-twin-cotorsion-pairs-and-model-structures}, \cite[Def.\ 2.1]{LiuYNakaoka-hearts-of-twin-cotorsion-pairs-on-extriangulated-categories} 
Let $\CU, \CV$ be full, additive subcategories of $\CB$ that are closed under taking direct summands. We call $(\CU,\CV)$ a \emph{cotorsion pair (on $\CB$)} if 
\begin{enumerate}[label=(\roman*)]
	\item $\BE(\CU,\CV) = 0$;
	\item $\CB = \Cone(\CV,\CU)$; and 
	\item $\CB = \CoCone(\CV,\CU)$. 
\end{enumerate}
\end{defn}
\begin{rem}\label{rem:U-V-extension-closed-and-proj-in-U-inj-in-V}
Let $\CU,\CV\subseteq\CB$ be full, additive subcategories of $\CB$ closed under taking direct summands. 
We denote by $\CU*\CV$ the full subcategory of $\CB$ consisting of objects $X\in\CB$
for which there is a conflation 
$\begin{tikzcd}[column sep =0.7cm]
U \arrow[tail]{r} & X \arrow[two heads]{r} & V
\end{tikzcd}$
in $\CB$ for some $U\in\CU$, $V\in\CV$; see \cite[p.\ 104]{LiuYNakaoka-hearts-of-twin-cotorsion-pairs-on-extriangulated-categories}. 
Moreover, if $(\CU,\CV)$ is a cotorsion pair, then $\CU$ is \emph{extension-closed}, i.e.\ $\CU * \CU \subseteq \CU$, and, similarly, $\CV$ is also extension-closed; see \cite[Rem.\ 4.6]{NakaokaPalu-extriangulated-categories-hovey-twin-cotorsion-pairs-and-model-structures}. 

\end{rem}
\begin{defn}\label{def:twin-cotorsion-pair}
\cite[Def.\ 4.12]{NakaokaPalu-extriangulated-categories-hovey-twin-cotorsion-pairs-and-model-structures} 
Let $(\CS,\CT)$ and $(\CU,\CV)$ be cotorsion pairs on $\CB$. 
Then $((\CS,\CT),(\CU,\CV))$ is called a \emph{twin cotorsion pair (on $\CB$)} if 
$\BE(\CS,\CV)=0$, or equivalently $\CS\subseteq\CU$, or equivalently $\CV\subseteq\CT$.
\end{defn}
For a subcategory $\CA\subseteq\CB$ that is closed under finite direct sums, we denote by $[\CA]$ the two-sided ideal of $\CB$ such that $[\CA](X,Y)$ consists of all the morphisms in $\CB(X,Y)$ that factor through an object lying in $\CA$. 
\begin{defn}\label{def:heart-and-subcategories-associated-to-twin-cotorsion-pair}
\cite[Def.\ 2.5, Def.\ 2.6]{LiuYNakaoka-hearts-of-twin-cotorsion-pairs-on-extriangulated-categories} 
Suppose $((\CS,\CT),(\CU,\CV))$ is a twin cotorsion pair on $\CB$. 
We define full subcategories of $\CB$ as follows:
\[
\begin{array}{cccc}
\CW \deff \CT \cap \CU, & 
\CB^{-}\deff \CoCone(\CW,\CS), & 
\CB^{+}\deff \Cone(\CV,\CW), & 
\CH\deff\CB^{-}\cap\CB^{+}.\end{array}
\]

For $\CA\in\{\CB^{-},\CB^{+},\CH\}$, we define 
$\ol{\CA}$ to be the additive quotient $\CA / [\CW]$. 
In particular, the subfactor category $\heart = \CH/[\CW]$ of $\CB$ is known as the \emph{heart} of $((\CS,\CT),(\CU,\CV))$. 
\end{defn}
Later we refer to the following construction (and its dual) from \cite{LiuYNakaoka-hearts-of-twin-cotorsion-pairs-on-extriangulated-categories} several times. 
\begin{defn}\label{def:nearly-cokernel-of-morphism-in-heart}
\cite[Def.\ 2.24]{LiuYNakaoka-hearts-of-twin-cotorsion-pairs-on-extriangulated-categories} 
Suppose we have a morphism $f\colon A\to B$ in $\CB$ where $A\in\CB^{-}$. 
We define $C_{f}\in\CB$ and $c_{f}\colon B\to C_{f}$ as follows. 
Since $A\in\CB^{-}$, there is an $\BE$-triangle 
$\begin{tikzcd}[column sep=0.7cm]
A \arrow[tail]{r}{}&W^{A}\arrow[two heads]{r}{}&S^{A}\arrow[dashed]{r}{\delta} &{},
\end{tikzcd}$ 
with $W^{A}\in\CW, S^{A}\in\CS$. 
Then $f\colon A\to B$ induces a morphism $(f,1_{S^{A}})\colon \delta \to f_{*}\delta$ of extensions, which we may realise by
\[
\begin{tikzcd}
A \arrow[tail]{r}\arrow{d}{f}&W^{A}\arrow[two heads]{r}\arrow{d}&S^{A}\arrow[dashed]{r}{\delta} \arrow[equals]{d}&{}\\
B \arrow[tail]{r}{c_{f}}&C_{f}\arrow[two heads]{r}&S^{A}\arrow[dashed]{r}{f_{*}\delta} &{}
\end{tikzcd}
\]
\end{defn}
Liu and Nakaoka showed that the heart of a twin cotorsion pair on an extriangulated category always carries additional structure. 
%
\begin{thm}\label{thm:heart-is-semi-abelian}
\cite[Thm.\ 2.32]{LiuYNakaoka-hearts-of-twin-cotorsion-pairs-on-extriangulated-categories} 
The heart of a twin cotorsion pair on $\CB$ is semi-abelian.
\end{thm}
%
\section{A case when \texorpdfstring{$\heart$}{Hbar} is integral}\label{sec:the-case-when-H-is-integral}
Throughout this section, let $(\CB, \BE, \fs)$ be an extriangulated category and, in addition, let $((\CS,\CT),(\CU,\CV))$ be a twin cotorsion pair on $\CB$. 
\begin{prop}\label{prop:f-epic-in-heart-iff-C-f-in-U}
\emph{\cite[Cor.\ 2.26]{LiuYNakaoka-hearts-of-twin-cotorsion-pairs-on-extriangulated-categories}} 
Let $f\in\CH(A,B)$ be a morphism. Then $\ol{f}\in\heart(A,B)$ is an epimorphism if and only if the object $C_f$ (as defined in \emph{Definition \ref{def:nearly-cokernel-of-morphism-in-heart}}) lies in $\CU$.
\end{prop}

The next lemma is a unification of \cite[Lem.\ 5.3]{Nakaoka-twin-cotorsion-pairs} and \cite[Lem.\ 5.5]{LiuY-hearts-of-twin-cotorsion-pairs-on-exact-categories}.
\begin{lem}
\label{lem:LN-lem-2.31-PB-diagram-conflation-with-end-term-in-U-implies-a-is-epic}
\cite[Lem.\ 2.31]{LiuYNakaoka-hearts-of-twin-cotorsion-pairs-on-extriangulated-categories} 
Suppose 
\[
\begin{tikzcd}
A \arrow{r}{\ol{a}}\arrow{d}[swap]{\ol{b}}\PB{dr}& B\arrow{d}{\ol{c}} \\ C\arrow{r}[swap]{\ol{d}} & D 
\end{tikzcd}
\] 
is a pullback diagram in $\heart$. 
Suppose further that there is an object $X\in\CB^{-}$ and morphisms 
$x_{B}\colon X\to B$ and $x_{C}\colon X\to C$ in $\CB$, 
such that 
$\ol{c}\circ \ol{x_{B}} = \ol{d}\circ \ol{x_{C}}$ in $\ol{\CB^{-}}$ 
and that there is a conflation 
$\begin{tikzcd}[column sep=0.7cm]
X \arrow[tail]{r}{x_{B}}& B\arrow[two heads]{r}{}& U
\end{tikzcd}$ 
with $U\in\CU$. 
Then $\ol{a}\colon A\to B$ is an epimorphism in $\heart$.
\end{lem}
The following definition is a direct generalisation of the notions from the exact setting.
\begin{defn}\label{def:projectives-ProjB-has-enough-projectives-duals}
\cite[Def.\ 3.23, Def.\ 3.25]{NakaokaPalu-extriangulated-categories-hovey-twin-cotorsion-pairs-and-model-structures} 
An object $P\in\CB$ is said to be \emph{projective} if, for any conflation 
$\begin{tikzcd}[column sep=0.7cm]
A \arrow[tail]{r}{}& B\arrow[two heads]{r}{b} & C 
\end{tikzcd}$ 
and for any morphism 
$c\colon P\to C$, there exists a morphism $d\colon P \to B$ such that $bd=c$. 
The full subcategory of $\CB$ consisting of all projective objects is denoted by $\Proj\CB$. 
We say that $\CB$ \emph{has enough projectives} if, for each object $C\in\CB$, there exists a conflation 
$\begin{tikzcd}[column sep=0.7cm]
A \arrow[tail]{r}{}& P\arrow[two heads]{r}{p} & C
\end{tikzcd}$ 
with $P$ projective. 

\emph{Injective} objects, the full subcategory $\Inj\CB$ and $\CB$ \emph{has enough injectives} are  all defined dually.
\end{defn}
\begin{rem}\label{rem:projective-injective-definition-coincide-with-defs-for-exact-categories-triangulated-category-has-enough-of-them}
\begin{enumerate}[label=(\roman*)]
    \item As noted in \cite[Exam.\ 3.26]{NakaokaPalu-extriangulated-categories-hovey-twin-cotorsion-pairs-and-model-structures}, if $(\CB,\BE,\fs)$ is an exact category, then the notions in Definition \ref{def:projectives-ProjB-has-enough-projectives-duals} all coincide with the usual ones. 
    If instead $(\CB,\BE,\fs)$ is a triangulated category, then $\Proj\CB = \{0\} = \Inj\CB$, and $\CB$ has enough projectives and enough injectives.
    \item For any cotorsion pair $(\CX,\CY)$ on $\CB$, we have $\Proj \CB\subseteq \CX$ and $\Inj \CB \subseteq \CY$; see \cite[Rem.\ 2.2]{LiuYNakaoka-hearts-of-twin-cotorsion-pairs-on-extriangulated-categories}. 
\end{enumerate}

\end{rem}
The following is the main result of this section, and unifies \cite[Thm.\ 6.3]{Nakaoka-twin-cotorsion-pairs} and \cite[Thm.\ 6.2]{LiuY-hearts-of-twin-cotorsion-pairs-on-exact-categories}. 
It also improves the latter because here $\CB$ is not assumed to be Krull-Schmidt. 
\begin{thm}\label{thm:integral-heart}
Let $(\CB,\BE,\fs)$ be an extriangulated category with enough projectives and injectives. 
Suppose $((\CS,\CT),(\CU,\CV))$ is a twin cotorsion pair on $\CB$, such that $\CU\subseteq\CS * \CT$ and $\Proj\CB\subseteq\CW$ (respectively, $\CT\subseteq\CU * \CV$ and $\Inj\CB\subseteq\CW$). Then $\heart=\CH/[\CW]$ is left integral (respectively, right integral), and hence integral.
\end{thm}
\begin{proof}
Note that as $\heart$ is a semi-abelian category by Theorem \ref{thm:heart-is-semi-abelian}, we have that $\heart$ is left integral if and only if $\heart$ is right integral by Proposition \ref{prop:semi-abelian-category-is-left-integral-iff-right-integral}. 
Therefore, integrality of $\heart$ will follow from integrality on one side. 
We will show that $\CU\subseteq\CS * \CT$ and $\Proj\CB\subseteq\CW$ imply $\heart$ is left integral. 
Showing that $\heart$ is right integral if $\CT\subseteq\CU * \CV$ and $\Inj\CB\subseteq\CW$ is similar.

Suppose that $\CU\subseteq\CS * \CT$ and $\Proj\CB\subseteq\CW$, and that we have a pullback diagram 
\[
\begin{tikzcd}
A \arrow{r}{\ol{a}}\arrow{d}[swap]{\ol{b}}\PB{dr}& B\arrow{d}{\ol{c}} \\
C \arrow{r}[swap]{\ol{d}}& D
\end{tikzcd}
\] 
in $\heart$, where $\ol{d}$ is an epimorphism. 
As $C\in\heart\subseteq\CB^{-}$, we obtain a morphism 
\begin{equation}\label{eqn:obtaining-c-d}
    \begin{tikzcd}
    C\arrow[tail]{r}{} \arrow{d}{d}&W^{C} \arrow[two heads]{r}{} \arrow{d}&S^{C} \arrow[dashed]{r}{}\arrow[equals]{d}& {}\\
    D \arrow[tail]{r}{c_{d}} & C_{d} \arrow[two heads]{r}{e} & S^{C} \arrow[dashed]{r}{\delta_{c_{d}}}& {}
    \end{tikzcd}
\end{equation}
of $\BE$-triangles, as in Definition \ref{def:nearly-cokernel-of-morphism-in-heart}, 
where $W^{C}\in\CW, S^{C}\in\CS$. 
By Proposition \ref{prop:f-epic-in-heart-iff-C-f-in-U}, 
we have that $C_{d}$ belongs to $\CU$ as $\ol{d}$ is an epimorphism.
By assumption, we have $\CU\subseteq\CS*\CT$, so there is a conflation 
$\begin{tikzcd}[column sep=0.7cm]
S\arrow[tail]{r}{f} &C_{d} \arrow[two heads]{r}{g} &T, 
\end{tikzcd}$ 
where $S\in\CS$ and $T\in\CT$. 

There is a conflation 
$\begin{tikzcd}[column sep=0.7cm]
B\arrow[tail]{r}{}  & W^{B} \arrow[two heads]{r}{h} &S^{B}, 
\end{tikzcd}$ where $W^{B}\in\CW, S^{B}\in\CS$, because  $B\in\heart\subseteq\CB^{-}$; and there is also a conflation 
$\begin{tikzcd}[column sep=0.7cm]
K_{S^{B}}\arrow[tail]{r}{i} & P_{S^{B}}\arrow[two heads]{r}{j} &S^{B}, 
\end{tikzcd}$ 
with $P_{S^{B}}\in\Proj\CB\subseteq\CW$, as $\CB$ has enough projectives. 
Since $P_{S^{B}}$ is projective we see that $j$ factors through $h$, and we obtain a morphism 
\begin{equation}\label{eqn:k-star-of-delta-i}
    \begin{tikzcd}
    K_{S^{B}}\arrow[tail]{r}{i}\arrow[dotted]{d}{k} &P_{S^{B}} \arrow[two heads]{r}{j}\arrow[dotted]{d}{} & S^{B}\arrow[equals]{d}\arrow[dashed]{r}{\delta_{i}}& {}\\
    B\arrow[tail]{r}{} &W^{B} \arrow[two heads]{r}{h} & S^{B}\arrow[dashed]{r}{}& {}
    \end{tikzcd}
\end{equation}
of $\BE$-triangles by an application of (ET$3^{\op}$). 
Since
$\begin{tikzcd}[column sep=0.7cm]
K_{S^{B}}\arrow[tail]{r}{i} & P_{S^{B}}\arrow[two heads]{r}{j} &S^{B}, 
\end{tikzcd}$ 
is a conflation, by Proposition \ref{prop:exact-sequences-from-E-triangle} we have an exact sequence
\[
\begin{tikzcd}[column sep=0.8cm]
\CB(S^{B},T)\arrow{r}{}&\CB(P_{S^{B}},T)\arrow{r}{-\circ i}&\CB(K_{S^{B}},T)\arrow{r}{}&\BE(S^{B},T)=0,
\end{tikzcd}
\] 
where the last term vanishes because $(\CS,\CT)$ is a cotorsion pair.
Thus, there exists $l\colon P_{S^{B}}\to T$ such that $li=gc_{d}ck \colon K_{S^{B}} \to T$. 
Then $l$ factors through $g$ as $P_{S^{B}}$ is projective, so there exists $m\colon P_{S^{B}}\to C_{d}$ such that $gm=l$. 
Note that this implies $g(c_{d}ck-mi)=gc_{d}ck-gmi=li-li=0$, and hence $c_{d}ck-mi\colon K_{S^{B}}\to C_{d}$ must factor through $f\colon S\to C_{d}$ by Proposition \ref{prop:exact-sequences-from-E-triangle}. 
That is, there exists $n\colon K_{S^{B}}\to S$ such that $fn=c_{d}ck-mi$. 

Let 
$\begin{tikzcd}[column sep=0.7cm]
S\arrow[tail]{r}{} &Q \arrow[two heads]{r}{} & S^{B}
\end{tikzcd}$ 
be a realisation of $n_{*}\delta_{i}$. 
Notice that this implies $Q\in\CS$ as $\CS$ is extension-closed (see 
Remark \ref{rem:U-V-extension-closed-and-proj-in-U-inj-in-V}). 
By Proposition \ref{prop:variation-of-LN-prop-1-20-same-end-terms}, there is a morphism $q\colon P_{S^{B}}\to Q$ such that 
\[
\begin{tikzcd}
K_{S^{B}}\arrow[tail]{r}{i}\arrow{d}{n} & P_{S^{B}}\arrow[two heads]{r}{j} \arrow{d}{q}& S^{B} \arrow[dashed]{r}{\delta_{i}}\arrow[equals]{d}{}& {}\\
S\arrow[tail]{r}{p} & Q\arrow[two heads]{r}{} &S^{B} \arrow[dashed]{r}{n_{*}\delta_{i}}& {}
\end{tikzcd}
\] 
commutes and  
$\begin{tikzcd}[ampersand replacement=\&, column sep=1.2cm]
K_{S^{B}}\arrow[tail]{r}{\begin{psmallmatrix}-n \\ i\end{psmallmatrix}} \& S\oplus P_{S^{B}} \arrow[two heads]{r}{\begin{psmallmatrix}p & q\end{psmallmatrix}} \& Q
\end{tikzcd}$ 
is a conflation.

As $\CB$ has enough projectives, there is an $\BE$-triangle 
$\begin{tikzcd}[column sep=0.7cm]
K_{Q}\arrow[tail]{r}{r} & P_{Q}\arrow[two heads]{r}{s} &Q \arrow[dashed]{r}{\delta_{r}}& {}
\end{tikzcd}$ 
with $P_{Q}$ projective. 
Furthermore, as $P_{Q}\in\Proj\CB\subseteq\CW$ and $Q\in\CS$, we have $K_{Q}$ lies in 
$\CoCone(\CW,\CS)=\CB^{-}$. 
We have that $s$ factorises through 
$(\,p\;\,q\,)$ 
since $P_{Q}$ is projective, and hence we have a morphism 
\begin{equation}\label{eqn:t-star-delta-r}
    \begin{tikzcd}[ampersand replacement=\&]
    K_{Q}\arrow[tail]{r}{r} \arrow[dotted]{d}{t}\& P_{Q}\arrow[two heads]{r}{s} \arrow[dotted]{d}{\begin{psmallmatrix}u\\ v\end{psmallmatrix}}\& Q \arrow[dashed]{r}{\delta_{r}}\arrow[equals]{d}\& {}\\
    K_{S^{B}}\arrow[tail]{r}{\begin{psmallmatrix}-n \\ i\end{psmallmatrix}} \& S\oplus P_{S^{B}} \arrow[two heads]{r}{\begin{psmallmatrix}p & q\end{psmallmatrix}} \& Q\arrow[dashed]{r}{t_{*}\delta_{r}}\& {}\end{tikzcd}
\end{equation} 
of $\BE$-triangles by (ET$3^{\op}$). 
Note that $fn=c_{d}ck-mi$ implies 
\[
c_{d}(ck) = fn+mi = 
(\,-f\;\,m\,) \circ \begin{psmallmatrix}-n \\ i \end{psmallmatrix}. 
\] 
Thus, by applying \ref{ET3}, we have a morphism of $\BE$-triangles as follows: 
\begin{equation}\label{eqn:t-star-delta-r-to-delta-c-d}
    \begin{tikzcd}[ampersand replacement=\&]
    K_{S^{B}}\arrow[tail]{r}{\begin{psmallmatrix}-n \\ i\end{psmallmatrix}} \arrow{d}{ck}\& S\oplus P_{S^{B}} \arrow[two heads]{r}{\begin{psmallmatrix}p & q\end{psmallmatrix}} \arrow{d}{\begin{psmallmatrix}-f & m\end{psmallmatrix}}\& Q\arrow[dashed]{r}{t_{*}\delta_{r}}\arrow[dotted]{d}\& {}\\
    D\arrow[tail]{r}{c_{d}} \& C_{d}\arrow[two heads]{r}{e} \&S^{C} \arrow[dashed]{r}{\delta_{c_{d}}}\& {}
    \end{tikzcd}
\end{equation} 
Composing \eqref{eqn:t-star-delta-r} and \eqref{eqn:t-star-delta-r-to-delta-c-d}, 
we obtain the morphism 
\begin{equation}\label{eqn:delta-r-to-delta-c-d}
    \begin{tikzcd}[ampersand replacement=\&, column sep=1.5cm]
    K_{Q}\arrow[tail]{r}{r} \arrow{d}{ckt}\& P_{Q}\arrow[two heads]{r}{s} \arrow{d}{mv-fu}\&Q \arrow[dashed]{r}{\delta_{r}}\arrow{d}\& {}\\
    D\arrow[tail]{r}{c_{d}} \& C_{d}\arrow[two heads]{r}{e} \&S^{C} \arrow[dashed]{r}{\delta_{c_{d}}}\& {}
    \end{tikzcd}
\end{equation} 

There is an $\BE$-triangle 
$\begin{tikzcd}[column sep=0.7cm]
K_{S^{C}}\arrow[tail]{r}{x} &P_{S^{C}} \arrow[two heads]{r}{y} &S^{C} \arrow[dashed]{r}{\delta_{x}}& {}
\end{tikzcd}$ 
with $P_{S^{C}}\in\Proj\CB$, as $\CB$ has enough projectives. 
Using that $P_{S^{C}}$ is projective and (ET$3^{\op}$), there is a morphism 
\begin{equation}\label{eqn:z-star-delta-x}
    \begin{tikzcd}
    K_{S^{C}}\arrow[tail]{r}{x} \arrow[dotted]{d}{z}&P_{S^{C}} \arrow[two heads]{r}{y} \arrow[dotted]{d}{}&S^{C} \arrow[dashed]{r}{\delta_{x}}\arrow[equals]{d}& {}\\
    C\arrow[tail]{r}{} &W^{C} \arrow[two heads]{r}{} &S^{C} \arrow[dashed]{r}{}& {}\\
    \end{tikzcd}
\end{equation} 
of $\BE$-triangles. Then composing morphisms \eqref{eqn:z-star-delta-x} and \eqref{eqn:obtaining-c-d}, we obtain a morphism of $\BE$-triangles as follows: 
\begin{equation}\label{eqn:delta-x-to-delta-c-d}
    \begin{tikzcd}
    K_{S^{C}}\arrow[tail]{r}{x} \arrow{d}{dz}&P_{S^{C}} \arrow[two heads]{r}{y} \arrow{d}{a'}&S^{C} \arrow[dashed]{r}{\delta_{x}}\arrow[equals]{d}& {}\\
    D\arrow[tail]{r}{c_{d}} & C_{d}\arrow[two heads]{r}{e} &S^{C} \arrow[dashed]{r}{\delta_{c_{d}}}& {}
    \end{tikzcd}
\end{equation} 

Consider the morphism $e(mv-fu)\colon P_{Q}\to S^{C}$. 
As 
$\begin{tikzcd}[column sep=0.7cm]
K_{S^{C}}\arrow[tail]{r}{x} &P_{S^{C}} \arrow[two heads]{r}{y} &S^{C} \arrow[dashed]{r}{\delta_{x}}& {}
\end{tikzcd}$ 
is an $\BE$-triangle and $P_{Q}$ is projective, there exists $b'\colon P_{Q}\to P_{S^{C}}$ such that $yb'=e(mv-fu)$. 
This yields $y(b'r)=e(mv-fu)r=ec_{d}ckt=0$ (using the commutativity of \eqref{eqn:delta-r-to-delta-c-d}), and so by Proposition \ref{prop:exact-sequences-from-E-triangle} there exists $c'\colon K_{Q}\to K_{S^{C}}$ such that $xc'=b'r$. 
In addition, we also see that 
$e(a'b'-(mv-fu))=ea'b'-e(mv-fu)=yb'-yb'=0$. 
Thus, there exists $d'\colon P_{Q}\to D$ such that $c_{d}d'=a'b'-(mv-fu)$, by Proposition \ref{prop:exact-sequences-from-E-triangle}, using the conflation 
$\begin{tikzcd}[column sep=0.7cm]
D\arrow[tail]{r}{c_{d}} & C_{d}\arrow[two heads]{r}{e} &S^{C}.
\end{tikzcd}$ 
Hence, we have
\begin{equation}\label{eqn:equality-involving-c-d}
\begin{array}{rcl}
     c_{d}dzc'  & = & a'xc' \\
                & = & a'b'r \\
                & = & c_{d}d'r+(mv-fu)r \\
                & = & c_{d}d'r + c_{d}ckt \\
                & = & c_{d}(d'r + ckt). \\
\end{array}
\end{equation}

As \eqref{eqn:delta-x-to-delta-c-d} is a morphism of $\BE$-triangles, we have 
$\delta_{c_{d}}=(1_{S^{C}})^{*}\delta_{c_{d}}=(dz)_{*}\delta_{x},$ and so 
$(dz)_{*}\delta_{x}$ is realised by the conflation 
$\begin{tikzcd}[column sep=0.7cm]
D\arrow[tail]{r}{c_{d}} & C_{d}\arrow[two heads]{r}{e} &S^{C}.
\end{tikzcd}$ 
Therefore, by Proposition \ref{prop:variation-of-LN-prop-1-20-same-end-terms}, there is a conflation 
\begin{equation}\label{eqn:extension-C-d-by-K-S-C}
\begin{tikzcd}[ampersand replacement=\&, column sep=1.5cm]
K_{S^{C}}\arrow[tail]{r}{\begin{psmallmatrix} -dz \\ x\end{psmallmatrix}} \& D\oplus P_{S^{C}}\arrow[two heads]{r}{\begin{psmallmatrix} c_{d}&e'\end{psmallmatrix}} \& C_{d}
\end{tikzcd}
\end{equation}
for some $e'\colon P_{S^{C}}\to C_{d}$ satisfying $e'x=c_{d}dz$. 
Consider the morphism 
$\begin{psmallmatrix}
-(d'r + ckt)\\
b'r
\end{psmallmatrix}
\colon K_{Q}\to D\oplus P_{S^{C}}$, and note that 
\[
\begin{array}{rclr}
			(\,c_{d}\;\,e'\,) \circ \begin{psmallmatrix}-(d'r + ckt)\\b'r\end{psmallmatrix}  
                & = & e'b'r-c_{d}(d'r + ckt) &\\
                & = & e'xc'-c_{d}(d'r + ckt) &\\
                & = & c_{d}dzc'-c_{d}(d'r + ckt) &\\
                & = & 0 & \text{by }\eqref{eqn:equality-involving-c-d}.\\
\end{array}
\] 
Thus, there exists $f'\colon K_{Q}\to K_{S^{C}}$ such that 
$\begin{psmallmatrix}-dzf'\\xf'\end{psmallmatrix} 
=\begin{psmallmatrix}-dz\\x\end{psmallmatrix}f'
=\begin{psmallmatrix}-(d'r + ckt)\\b'r\end{psmallmatrix}.$ 
In particular, we see that 
\begin{equation}\label{eqn:d-x-C-equals-c-x-B-mod-W}
    d(zf')=d'r+c(kt).
\end{equation}

From \eqref{eqn:k-star-of-delta-i} we get a conflation
\begin{equation}\label{eqn:conflation-starting-K-S-B}
    \begin{tikzcd}[ampersand replacement=\&, column sep=1.2cm]
    K_{S^{B}}\arrow[tail]{r}{\begin{psmallmatrix}-k \\ i\end{psmallmatrix}} \& B\oplus P_{S^{B}}\arrow[two heads]{r}{} \& W^{B}
    \end{tikzcd}
\end{equation} 
and from \eqref{eqn:t-star-delta-r} we get a conflation
\begin{equation}\label{eqn:conflation-starting-K-Q}
    \begin{tikzcd}[ampersand replacement=\&, column sep=1.2cm]
    K_{Q}\arrow[tail]{r}{\begin{psmallmatrix}-t\\r\end{psmallmatrix}} \&K_{S^{B}}\oplus P_{Q} \arrow[two heads]{r}{} \& S\oplus P_{S^{B}},
    \end{tikzcd}
\end{equation} 
using Proposition \ref{prop:variation-of-LN-prop-1-20-same-end-terms}. 
As \eqref{eqn:conflation-starting-K-S-B} is a conflation, we also have a conflation
\begin{equation}\label{eqn:conflation-starting-K-S-B-oplus-P-Q}
    \begin{tikzcd}[ampersand replacement=\&, column sep=1.7cm, row sep=2cm]
    K_{S^{B}}\oplus P_{Q}\arrow[tail]{r}{
    \begin{psmallmatrix}
    -k & 0 \\ 
    i & 0 \\
    0 & 1_{P_{Q}}
    \end{psmallmatrix}
    } \& B\oplus P_{S^{B}}\oplus P_{Q}\arrow[two heads]{r}{} \& W_{1}
    \end{tikzcd}
\end{equation} 
using \ref{ET2}. Then, applying \ref{ET4} to the conflations \eqref{eqn:conflation-starting-K-Q} and \eqref{eqn:conflation-starting-K-S-B-oplus-P-Q}, we have a commutative diagram 
\[
\begin{tikzcd}[ampersand replacement=\&,column sep=1cm, row sep=1.3cm]
    K_{Q}\arrow[tail]{r}{\begin{psmallmatrix}-t\\r\end{psmallmatrix}}\arrow[equals]{d} \& K_{S^{B}}\oplus P_{Q} \arrow[two heads]{r}{}\arrow[tail]{d}{\begin{psmallmatrix}
    -k & 0 \\ 
    i & 0 \\
    0 & 1_{P_{Q}}
    \end{psmallmatrix}} \& S\oplus P_{S^{B}}\arrow[tail]{d}{}\\
K_{Q}\arrow[tail]{r}{x_{B}} \& B\oplus P_{S^{B}}\oplus P_{Q}\arrow[two heads]{r}{}\arrow[two heads]{d}{} \& M \arrow[two heads]{d}{} \\
 \& W^{B}\arrow[equals]{r}{} \& W^{B}
\end{tikzcd}
\]
in which 
$x_{B}\deff \begin{psmallmatrix}kt \\ -it \\ r\end{psmallmatrix}$. 
Note that $M$ lies in $\CU$ as both $W^{B}$ and $S\oplus P_{S^{B}}$ lie in the extension-closed subcategory $\CU$.

Since $B\in\CH$ and $P_{S^{B}}, P_{Q}\in\Proj\CB\subseteq\CW\subseteq\CH$, we have that $B\oplus P_{S^{B}}\oplus P_{Q}$ is an object in $\CH$. Thus, consider the following commutative diagram
\[
\begin{tikzcd}
A \arrow{r}{\ol{\iota_{B}a}}\arrow[equals]{d}& B\oplus P_{S^{B}}\oplus P_{Q}\arrow{d}{\ol{\pi_{B}}} \\
A \arrow{r}{\ol{a}}\arrow{d}[swap]{\ol{b}}\PB{dr}& B\arrow{d}{\ol{c}} \\
C \arrow{r}[swap]{\ol{d}}& D
\end{tikzcd}
\] 
in $\heart$, where $\iota_{B}\colon B \into B\oplus P_{S^{B}}\oplus P_{Q}$ is the canonical inclusion and 
$\pi_{B} = 
(\,1_{B}\;\, 0\;\,0\,)
\colon B\oplus P_{S^{B}}\oplus P_{Q} \onto B$ is the canonical projection. 
Notice that $\ol{\iota_{B}}$ and $\ol{\pi_{B}}$ are mutually inverse in $\heart$ as $P_{S^{B}}\oplus P_{Q}\in\CW$. 
Hence the square 
\[
\begin{tikzcd}
A \arrow{r}{\ol{\iota_{B}a}}\arrow{d}[swap]{\ol{b}}& B\oplus P_{S^{B}}\oplus P_{Q}\arrow{d}{\ol{c}\ol{\pi_{B}}} \\
C \arrow{r}[swap]{\ol{d}}& D
\end{tikzcd}
\] 
is also a pullback square in $\heart$. Setting $x_{C}\deff zf'\colon K_{Q}\to C$, we see that 
\[
\begin{array}{rcll}
    \ol{d}\circ\ol{x_{C}}   & = & \ol{dzf'} & \\
                            & = & \ol{d'r}+\ol{ckt} & \text{by }\eqref{eqn:d-x-C-equals-c-x-B-mod-W}\\
                            & = & \ol{ckt} & \text{as }\ol{d'r}=0\\
                            & = & 
                                \ol{c}\circ \ol{(\,1_{B}\;\, 0\;\,0\,)}\circ
                                \ol{\begin{psmallmatrix}kt \\ -it \\ r\end{psmallmatrix}} & \\
                            & = & \ol{c\pi_{B}}\circ \ol{x_{B}}. & \\
\end{array}
\] 
Therefore, by Lemma \ref{lem:LN-lem-2.31-PB-diagram-conflation-with-end-term-in-U-implies-a-is-epic}, we conclude that $\ol{\iota_{B}a}$ is an epimorphism. Finally, $\ol{a}$ is an epimorphism since $\ol{\iota_{B}}$ is an isomorphism in $\heart$, and we are done.
\end{proof}
\begin{rem}\label{rem:difference-in-our-proof-to-exact-and-triangulated-cases}
Theorem \ref{thm:integral-heart} unifies the analogous results for triangulated and exact categories. However, the proof we give here differs in several aspects. 
We note that our proof is not a direct generalisation of the proof for triangulated categories. 
This is because an extriangulated category does not come equipped with a suspension/shift functor. 
One way to work around this is to use that the extriangulated category has enough projectives, as one would do in the exact category case, in order to obtain what would be a negative shift of an object. 
Thus, our proof is inspired by the exact case. 
But the proof in \cite{LiuY-hearts-of-twin-cotorsion-pairs-on-exact-categories} uses the defining property of a monomorphism, which we cannot exploit in the extriangulated setting. 
This is a key difference between our proof above and the proof for exact categories.
\end{rem}
%
\section{A case when \texorpdfstring{$\heart$}{Hbar} is quasi-abelian}\label{sec:the-case-when-H-is-quasi-abelian}
%
In this section we give an analogue of \cite[Thm.\ 3.4]{Shah-quasi-abelian-hearts-of-twin-cotorsion-pairs-on-triangulated-cats} for the extriangulated setting, which also improves \cite[Thm.\ 7.4]{LiuY-hearts-of-twin-cotorsion-pairs-on-exact-categories}. 
First, let us recall a key lemma from \cite{Shah-quasi-abelian-hearts-of-twin-cotorsion-pairs-on-triangulated-cats}.
\begin{lem}
\label{lem:Shah-lemma-3.1}
\cite[Lem.\ 3.1]{Shah-quasi-abelian-hearts-of-twin-cotorsion-pairs-on-triangulated-cats} 
Let
$\CA$ be a left semi-abelian category. Suppose 
\[
\begin{tikzcd} A \arrow{r}{a}\arrow{d}[swap]{b}\PB{dr}& B\arrow{d}{c} \\ C\arrow{r}[swap]{d} & D \end{tikzcd}
\] 
is
a pullback diagram in $\CA$. Suppose we also have morphisms $x_{B}\colon X\to B$
and $x_{C}\colon X\to C$, such that $x_{B}$ is a cokernel and $cx_{B} = dx_{C}$.
Then $a\colon A\to B$ is also a cokernel in $\CA$.
\end{lem}
\begin{thm}\label{thm:quasi-abelian-heart}
Let $(\CB,\BE,\fs)$ be an extriangulated category. Suppose $((\CS,\CT),(\CU,\CV))$ is a twin cotorsion pair on $\CB$. If $\CH=\CB^{-}$ or $\CH=\CB^{+}$, then $\heart=\CH/[\CW]$ is quasi-abelian.
\end{thm}
\begin{proof}
The heart $\heart$ is semi-abelian by Theorem \ref{thm:heart-is-semi-abelian}, so
we have that $\heart$ is left quasi-abelian if and only if
$\heart$ is right quasi-abelian by Proposition \ref{prop:semi-abelian-category-is-left-quasi-abelian-iff-right-quasi-abelian}.
Therefore, we will show that if $\CH=\CB^{-}$ then $\heart$
is left quasi-abelian. Showing $\heart$ is right quasi-abelian
whenever $\CH=\CB^{+}$ is similar.

Suppose $\CH=\CB^{-}$ and that we have a pullback diagram 
\[
\begin{tikzcd}
A \arrow{r}{\ol{a}}\arrow{d}[swap]{\ol{b}}\PB{dr}& B\arrow{d}{\ol{c}} \\
C \arrow{r}[swap]{\ol{d}}& D
\end{tikzcd}
\]
in $\heart$, where $\ol{d}$ is a cokernel. 
By \cite[Lem.\ 2.28]{LiuYNakaoka-hearts-of-twin-cotorsion-pairs-on-extriangulated-categories},
we may assume that, up to isomorphism in $\heart$, the morphism  $d\in\CH(C,D)$ fits into a conflation 
$\begin{tikzcd}[column sep=0.7cm]
C\arrow[tail]{r}{d}& D\arrow[two heads]{r}{}& S
\end{tikzcd}$ 
in $\CB$ with $S\in\CS$. 
Furthermore, we know $D\in\CH=\CB^+ \cap \CB^-\subseteq \CB^+$, and so by \cite[Lemma 2.8(1)]{LiuYNakaoka-hearts-of-twin-cotorsion-pairs-on-extriangulated-categories} there exist $W\in\CW$ and $w\colon W\to D$ giving a deflation $(\,c\;\,w\,)\colon B\oplus W\onto D$ in $\CB$. 
Then, in $\heart$, we have $B\oplus W \iso B$ and, up to isomorphism, $\ol{(\,c\;\,w\,)} = \ol{c}$. 
Thus, without loss of generality, we may replace $c$ by $(\,c\;\,w\,)$ and $B$ by $B\oplus W$. 
That is, we may assume $c\colon B\to D$ is part of a conflation 
$\begin{tikzcd}[column sep=0.7cm]
B'\arrow[tail]{r}{}& B\arrow[two heads]{r}{c}& D.
\end{tikzcd}$

Applying (ET$4^{\op}$) to the conflations 
$\begin{tikzcd}[column sep=0.7cm]
C\arrow[tail]{r}{d}& D\arrow[two heads]{r}{}& S
\end{tikzcd}$ 
and 
$\begin{tikzcd}[column sep=0.7cm]
B'\arrow[tail]{r}{}& B\arrow[two heads]{r}{c}& D,
\end{tikzcd}$ 
we obtain a commutative diagram 
\[
\begin{tikzcd}
B' \arrow[equals]{r}\arrow[tail]{d}{}& B'\arrow[tail]{d}{} &&\\
X \arrow[tail]{r}{x_B}\arrow[two heads]{d}[swap]{x_C}& B\arrow[two heads]{r}{}\arrow[two heads]{d}{c}& S\arrow[equals]{d} \arrow[dashed]{r}{\delta}& {}\\
C \arrow[tail]{r}{d}&D\arrow[two heads]{r}{} &S&
\end{tikzcd}
\]
in which each row and column is a conflation. 
Note that, by \cite[Lem.\ 2.9(b)]{LiuYNakaoka-hearts-of-twin-cotorsion-pairs-on-extriangulated-categories}, $X\in\CB^-=\CH$ since $B\in\CB^-$. 
Therefore, we have a commutative square 
\[
\begin{tikzcd}X \arrow{r}{\ol{x_B}}\arrow{d}[swap]{\ol{x_C}}& B\arrow{d}{\ol{c}}\\
C \arrow{r}{\ol{d}}&D
\end{tikzcd}
\]
in $\heart$, and so it is enough to show $\ol{x_B}$ is a cokernel in $\heart$ by Lemma \ref{lem:Shah-lemma-3.1}.

Dually to Definition \ref{def:nearly-cokernel-of-morphism-in-heart}, 
we obtain a morphism $k_{x_B}\colon K_{x_B}\to X$ in $\CB$ as in the following commutative diagram:
\[
\begin{tikzcd}[column sep=1.3cm]
V\arrow[tail]{r}{}\arrow[equals]{d}&K_{x_B}\arrow{d}{e}\arrow[two heads]{r}{k_{x_B}}&X\arrow[tail]{d}{x_B}\arrow[dashed]{r}{(x_{B})^{*}\delta'}& {}\\
V\arrow[tail]{r}{v}&W'\arrow[two heads]{r}{w'}&B \arrow[dashed]{r}{\delta'}& {}
\end{tikzcd}
\]
Using (ET$4^{\op}$) we obtain a commutative diagram
\[
\begin{tikzcd}
V\arrow[equals]{r}\arrow[tail]{d}&V\arrow[tail]{d}{v}&\\
P\arrow[tail]{r}{}\arrow[two heads]{d}{p}&W'\arrow[two heads]{r}{}\arrow[two heads]{d}{w'}&S\arrow[equals]{d}\arrow[dashed]{r}{\delta''}& {}\\
X\arrow[dashed]{d}{(x_{B})^{*}\delta'}\arrow[tail]{r}{x_B}&B\arrow[dashed]{d}{\delta'}\arrow[two heads]{r}{s}&S\arrow[dashed]{r}{\delta}& {}\\
{}&{}&&\\
\end{tikzcd}
\] 
in which $p_{*}\delta''=\delta$. 
Thus, the conflations 
$\begin{tikzcd}[column sep=0.9cm]
V\arrow[tail]{r}{}&K_{x_{B}}\arrow[two heads]{r}{k_{x_{B}}}& X
\end{tikzcd}$ 
and 
$\begin{tikzcd}[column sep=0.7cm]
V\arrow[tail]{r}{}&P\arrow[two heads]{r}{p}& X
\end{tikzcd}$ 
both realise the extension $(x_{B})^{*}\delta'$, and hence are equivalent. 
This yields an isomorphism $q\colon K_{x_{B}}\to P$, such that $pq=k_{x_{B}}$, and also a morphism of $\BE$-triangles 
\[
\begin{tikzcd}
K_{x_{B}}\arrow[tail]{r}{f}\arrow{d}{\iso}[swap]{q} & W'\arrow[two heads]{r}{}\arrow[equals]{d}&\arrow[equals]{d}S \arrow[dashed]{r}{\eps}& {} \\
P\arrow[tail]{r}{} & W'\arrow[two heads]{r}{}&S \arrow[dashed]{r}{\delta''}& {} 
\end{tikzcd}
\]
where $q_{*}\eps =\delta''$. 
This implies that there is also a morphism   
\[
\begin{tikzcd}
K_{x_{B}}\arrow[tail]{r}{f}\arrow{d}{k_{x_{B}}}&W'\arrow[two heads]{r}{}\arrow{d}{w'}& S\arrow[equals]{d}\arrow[dashed]{r}{\eps}&{}\\
X\arrow[tail]{r}{x_{B}}&B\arrow[two heads]{r}{s}& S\arrow[dashed]{r}{\delta}&{}
\end{tikzcd}
\] 
of $\BE$-triangles, since  $(k_{x_{B}})_{*}\eps=p_{*}q_{*}\eps=p_{*}\delta''=\delta$. Furthermore, we also see that $K_{x_{B}}\in\CoCone(\CW,\CS)=\CB^{-}=\CH$.

We claim that $\ol{x_{B}}\colon X\to B$ is a cokernel of $\ol{k_{x_{B}}}\colon K_{x_{B}}\to X$ in $\heart$. 
We will first show that $\ol{x_{B}}$ is a weak cokernel of $\ol{k_{x_{B}}}$, and secondly that $\ol{x_{B}}$ is an epimorphism; this is enough by, for example, \cite[Lem.\ 2.5]{BuanMarsh-BM2}.

Note that $x_{B}\circ k_{x_{B}}= w'\circ f$ factors through $\CW$, so we have $\ol{x_{B}}\circ \ol{k_{x_{B}}}=0$ in $\heart$. 
Now suppose that there is $\ol{g}\colon X\to Y$ in $\heart$ such that $\ol{g}\circ \ol{k_{x_{B}}}=0$ in $\heart$. 
Then $g\circ k_{x_{B}}\colon K_{x_{B}}\to Y$ factors through $\CW$. Thus, there exists a commutative square 
\[
\begin{tikzcd}
K_{x_{B}} \arrow{r}{k_{x_{B}}}\arrow{d}{h}& X\arrow{d}{g}\\
W''\arrow{r}{i} & Y
\end{tikzcd}
\]
in $\CB$, with $W''\in\CW$. 

Since 
$\begin{tikzcd} [column sep=0.7cm]
X\arrow[tail]{r}{x_B}&B\arrow[two heads]{r}{s}&S\arrow[dashed]{r}{\delta}&{}
\end{tikzcd}$ 
is an $\BE$-triangle, by Proposition \ref{prop:exact-sequences-from-E-triangle} we have an exact sequence 
\begin{equation}\label{eqn:exact-sequence-induced-from-delta}
\begin{tikzcd}[column sep=1.5cm]
\CB(B,Y)\arrow{r}{\CB(x_{B},Y)}& \CB(X,Y)\arrow{r}{(\delta^{\sharp})_{Y}} & \BE(S,Y)\arrow{r}{\BE(s,Y)}& \BE(B,Y),
\end{tikzcd}
\end{equation}
where 
$(\delta^{\sharp})_{Y}:\CB(X,Y)\rightarrow \BE(S,Y)$ 
is given by 
$(\delta^{\sharp})_{Y}(r)= r_{*}\delta$. 
Note that  
\[
(\delta^{\sharp})_{Y}(g) = g_{*}(\delta) = g_{*}(k_{x_{B}})_{*}\eps = (gk_{x_{B}})_{*}\eps = (ih)_{*}\eps = i_{*}h_{*}\eps = 0,
\] 
as $h_{*}\eps\in \BE(S,W'')=0$ because $\CW\subseteq\CT$. 
Thus, $g$ is in the kernel of the morphism $(\delta^{\sharp})_{Y}$ and so, by the exactness of \eqref{eqn:exact-sequence-induced-from-delta}, there exists $j\colon B\to Y$ such that $j x_{B}=g$.
%
%
%
As $B,Y\in\CH$, we have $j\in\CH(B,Y)$ and $\ol{g}=\ol{j}\circ\ol{x_{B}}$. 
Hence, $\ol{x_{B}}$ is a weak cokernel for $\ol{k_{x_{B}}}$ in $\heart$.

Lastly, note that for any $V\in\CV$ we have, by Proposition \ref{prop:exact-sequences-from-E-triangle}, an exact sequence 
\[
\begin{tikzcd}[column sep=1.5cm]
0=\BE(S,V)\arrow{r}{} & \BE(B,V) \arrow{r}{\BE(x_{B},V)} &\BE(X,V),
\end{tikzcd}
\]
where $\BE(S,V)=0$, as $\CS\subseteq \CU$ and $(\CU,\CV)$ is a cotorsion pair. 
That is, $\BE(x_{B},V)$ is monomorphic for any $V\in\CV$, and therefore $\ol{x_{B}}$ is an epimorphism in $\heart$ by \cite[Prop.\ 2.29]{LiuYNakaoka-hearts-of-twin-cotorsion-pairs-on-extriangulated-categories}. 
Hence, $\ol{x_{B}}$ is the cokernel of $\ol{k_{x_{B}}}$ in the (left) semi-abelian category $\heart$, so $\ol{a}$ is also a cokernel in $\heart$ by Lemma \ref{lem:Shah-lemma-3.1} 
and we are done.
\end{proof}
The next corollary gives a unification of \cite[Thm.\ 7.4]{LiuY-hearts-of-twin-cotorsion-pairs-on-exact-categories} and \cite[Cor.\ 3.5]{Shah-quasi-abelian-hearts-of-twin-cotorsion-pairs-on-triangulated-cats}, 
and follows immediately from Theorems \ref{thm:integral-heart} and \ref{thm:quasi-abelian-heart}.
\begin{cor}\label{cor:heart-is-integral-and-quasi-abelian-if-U-in-T-or-T-in-U}
Let $(\CB,\BE,\fs)$ be an extriangulated category with enough projectives and injectives. 
Suppose $((\CS,\CT),(\CU,\CV))$  is a twin cotorsion pair on $\CB$. 
If $\CT\subseteq\CU$ or $\CU\subseteq \CT$, then $\heart=\CH/[\CW]$ is integral and quasi-abelian.
\end{cor}
%
\section{Localisation of an integral heart}
\label{sec:localisation-of-an-integral-heart}
%
%
In this section, we fix an extriangulated category $(\CB,\BE,\fs)$ with enough projectives and injectives. 
We also suppose that there is a twin cotorsion pair $((\CS,\CT),(\CU,\CV))$ on $\CB$ with $\CT=\CU$. 
Note that for this twin cotorsion pair we have $\CW=\CT=\CU$, $\CB^{+}=\CB=\CB^{-}$ and so its heart is $\heart=\CB/[\CW]$. 
By Corollary \ref{cor:heart-is-integral-and-quasi-abelian-if-U-in-T-or-T-in-U}, $\heart$ is integral (and quasi-abelian), and hence the class $\CR$ of regular morphisms in $\heart$ admits a \emph{calculus of left fractions} (see \cite[\S I.2]{GabrielZisman-calc-of-fractions}) by \cite[Prop.\ 6]{Rump-almost-abelian-cats}. 
(This also implies that $\heart_{\CR}$ is an abelian category by \cite[Thm.\ 4.8]{BuanMarsh-BM2}.) 
Thus, the objects of the localisation $\heart_{\CR}$ are the objects of $\heart$, and a morphism $X\to Y$ in $\heart_{\CR}$ is a \emph{left fraction} $[\ol{f},\ol{r}]_{\LF}$ of the form 
\[
\begin{tikzcd}
X \arrow{r}{\ol{f}}& A & Y\arrow{l}[swap]{\ol{r}}, 
\end{tikzcd}
\] 
up to a certain equivalence, where $\ol{f}\in\heart(X,A)$ and $\ol{r}\in\CR$ (see \cite[\S I.2]{GabrielZisman-calc-of-fractions} for details). 
The localisation functor $L_{\CR}\colon \heart\to\heart_{\CR}$ maps a morphism $\ol{f}\colon X\to A$ in $\heart$ to $L_{\CR}(\ol{f})=[\ol{f},\ol{1_{A}}]_{\LF}$. 
In particular, any morphism $\ol{r}\colon Y\to A$ in the class $\CR$ of regular morphisms in 
$\heart$ is mapped to $L_{\CR}(\ol{r})=[\ol{r},\ol{1_{Y}}]_{\LF}$, which is invertible with inverse $\tensor[]{[\ol{r},\ol{1_{Y}}]}{_{\LF}^{-1}}=[\ol{1_{A}},\ol{r}]_{\LF}$. 
Furthermore, $L_{\CR}\colon\overline{\CH}\to\overline{\CH}_{\CR}$
is an additive functor; see \cite[Rem.\ 4.3]{BuanMarsh-BM2}.

Let us denote by $\ol{\SH}\deff\ol{\SH}_{(\CS,\CT)}$ the heart $\CoCone(\CS,\CS)/[\CS]$ of the \emph{degenerate} twin cotorsion pair $((\CS,\CT),(\CS,\CT))$. 
The category $\ol{\SH}$ is also the heart of the \emph{single} cotorsion pair $(\CS,\CT)$; see \cite{LiuYNakaoka-hearts-of-twin-cotorsion-pairs-on-extriangulated-categories}. 
In this section, we will show that there is an equivalence $\heart_{\CR}\simeq\ol{\SH}$, giving an analogue of \cite[Thm.\ 4.8]{Shah-quasi-abelian-hearts-of-twin-cotorsion-pairs-on-triangulated-cats} for the extriangulated setting; see Theorem \ref{thm:T-is-U-implies-localisation-of-H-at-regular-morphisms-is-equivalent-to-heart-of-cotorsion-pair-S-T}. 
However, we note that this section improves some results from \cite[\S 4]{Shah-quasi-abelian-hearts-of-twin-cotorsion-pairs-on-triangulated-cats} since we make no Krull-Schmidt assumption in this article.

Let $\iota\colon \SH=\CoCone(\CS,\CS)\to \CB=\CH$ be the canonical inclusion functor, and let $Q_{[\CS]}\colon \SH\to \ol{\SH}$ and $Q_{[\CW]}\colon \CH\to \heart$ be the canonical additive quotient functors. 
Note that since $\CS\subseteq\CU=\CW$ under our assumptions, any morphism in the ideal $[\CS]$ in $\SH$ vanishes under the composition $Q_{[\CW]}\circ\iota\colon \SH\to\heart$. Therefore, there is a unique additive functor $F\colon\overline{\SH}\to\overline{\CH}$ that makes the diagram of functors 
\[
\begin{tikzcd}
\SH
\arrow{r}{\iota}\arrow{d}[swap]{Q_{[\CS]}}\commutes{dr}& \CH \arrow{d}{Q_{[\CW]}} \\
\overline{\SH} \arrow[dotted]{r}[swap]{F}& \overline{\CH}
\end{tikzcd}
\] 
commute. 
In particular, $F$ is the identity on objects and maps the coset $f+[\CS](X,Y)$ to the coset $\ol{f}=f+[\CW](X,Y)$ for any morphism $f\colon X\to Y$ in $\CB$. 
Recall that a cotorsion pair $(\CU,\CV)$ is said to be \emph{rigid} if $\CU\subseteq\CV$. 
The next result follows from 
\cite[Prop.\ 3.3]{Liu-localisations-of-the-hearts-of-cotorsion-pairs}, 
noting that $(\CS,\CT)$ is rigid as $\CS\subseteq \CU=\CT$. 
Furthermore, this improves \cite[Lem.\ 4.3]{Shah-quasi-abelian-hearts-of-twin-cotorsion-pairs-on-triangulated-cats} since no Krull-Schmidt restriction is needed here. 
\begin{lem}
\label{lem:every-object-of-B-is-isomorphic-in-localisation-to-an-object-lying-in-CoConeS-S}
Let $X$ be an arbitrary object of $\CB$. Then there exists a conflation 
\[
\begin{tikzcd}
Z\arrow[tail]{r}{g} & Y\arrow[two heads]{r}{f} & X,
\end{tikzcd}
\] 
such that  $Y\in\CoCone(\CS,\CS)=\SH$, 
$Z\in\CW$ and $\ol{f}$ is a regular morphism in $\heart$.
\end{lem}
Set $G\deff L_{\CR}\circ F$. 
Then $G(X)=X$ and $G(f+[\CS](X,Y))=[f+[\CW](X,Y),1_{Y}]_{\LF}=[\ol{f},1_{Y}]_{\LF}$. 
We show that $G$ is an equivalence of categories over several steps in the remainder of this section. 
The next result is an analogue of \cite[Prop.\ 4.4]{Shah-quasi-abelian-hearts-of-twin-cotorsion-pairs-on-triangulated-cats}, and the proof easily generalises using Lemma \ref{lem:every-object-of-B-is-isomorphic-in-localisation-to-an-object-lying-in-CoConeS-S}, so we omit the proof here. 
\begin{prop}\label{prop:G-is-dense}
The functor $G\colon\overline{\mathscr{H}}\to\overline{\CH}_{\CR}$
is dense.
\end{prop}

We have an analogue of \cite[Lem.\ 4.5]{Shah-quasi-abelian-hearts-of-twin-cotorsion-pairs-on-triangulated-cats}.
\begin{lem}\label{lem:X-in-CoCone-S-S-and-f-X-to-Y-factors-through-W-implies-f-factors-through-S}
Suppose $X\in\CoCone(\CS,\CS)$ and $f\colon X\to Y$ is a morphism
in $\CB$. If $f$ factors through $\CW$, then $f$ factors through
$\CS$.
\end{lem}
\begin{proof}
Suppose $f\colon X\to Y$ factors as $f=ba$ for some $a\colon X\to W$ and $b\colon W\to Y$, where $W\in\CW$. 
Since $X\in\CoCone(\CS,\CS)$, there is an $\BE$-triangle
$\begin{tikzcd}[column sep=0.7cm]
X\arrow[tail]{r}{s} & S_{1} \arrow[two heads]{r}{}& S_{0} \arrow[dashed]{r}{\delta}& {},
\end{tikzcd}$ 
with $S_{0},S_{1}\in\CS$. 
By Proposition \ref{prop:exact-sequences-from-E-triangle}, there is an exact sequence 
\[
\begin{tikzcd}[column sep=1.3cm]
\CB(S_{1},Y) \arrow{r}{\CB(s,Y)}& \CB(X,Y)\arrow{r}{(\delta^{\sharp})_{Y}} & \BE(S_{0},Y),
\end{tikzcd}
\] 
where $(\delta^{\sharp})_{Y}(h)=h_{*}\delta$ for any $h\colon X\to Y$ in $\CB$. 
Note that $a_{*}\delta\in\BE(S_{0},W)=0$ because $((\CS,\CT),(\CU,\CV))$ is a twin cotorsion pair and $\CW=\CU$, so $(\delta^{\sharp})_{Y}(f)=f_{*}\delta=b_{*}a_{*}\delta=0$. 
Hence, there exists $g\colon S_{1}\to Y$ such that $gs=f$ and we see that $f$ factors through $\CS$.
\end{proof}
An analogue of \cite[Prop.\ 4.6]{Shah-quasi-abelian-hearts-of-twin-cotorsion-pairs-on-triangulated-cats} follows immediately, using the lemma above, hence we omit the proof.
\begin{prop}\label{prop:G-is-faithful}
The functor $G\colon\overline{\SH}\to\heart_{\CR}$
is faithful.
\end{prop}
For the next proposition, we have adapted methods from 
\cite{Liu-localisations-of-the-hearts-of-cotorsion-pairs}. 
%
\begin{prop}\label{prop:G-is-full}
The functor $G\colon\overline{\SH}\to\heart_{\CR}$
is full.
\end{prop}
\begin{proof}
Let $X,Y\in\ol{\SH}=\CoCone(\CS,\CS)/[\CS]$ and let 
$[\ol{f},\ol{r}]_{\LF}\colon 
\begin{tikzcd}[column sep=0.7cm]
X \arrow{r}{\ol{f}}& A & Y\arrow{l}[swap]{\ol{r}}
\end{tikzcd}$ 
be an arbitrary morphism in 
$\heart_{\CR}(GX,GY)=\heart_{\CR}(X,Y)$. 
Since $X\in\CoCone(\CS,\CS)$ and $\CB$ has enough projectives, there are conflations 
$\begin{tikzcd}[column sep=0.7cm]
X\arrow[tail]{r}{} &S_{1} \arrow[two heads]{r}{} & S_{0}
\end{tikzcd}$ 
and 
\begin{equation}\label{eqn:conflation-K1-P1-S1}
    \begin{tikzcd}
    K_{1}\arrow[tail]{r}{a} &P_{1} \arrow[two heads]{r}{} & S_{1},
    \end{tikzcd}
\end{equation} 
with $S_{0},S_{1}\in \CS$ and $P_{1}\in\Proj\CB$. 
By (ET$4^{\op}$), there is a commutative diagram 
\begin{equation}\label{eqn:liu-localisations-lem-2-9-applied-to-X}
    \begin{tikzcd}
    K_{1} \arrow[equals]{r}\arrow[tail]{d}[swap]{b}& K_{1}\arrow[tail]{d}{a}& \\
    K_{0}\arrow[two heads]{d}[swap]{d}\arrow[tail]{r}{c} & P_{1}\arrow[two heads]{d}{}\arrow[two heads]{r}& S_{0}\arrow[equals]{d} \\
    X \arrow[tail]{r}& S_{1}\arrow[two heads]{r}& S_{0}
    \end{tikzcd}
\end{equation}
of conflations. 

As in the dual of Definition \ref{def:nearly-cokernel-of-morphism-in-heart}, 
we have a commutative diagram 
\[
\begin{tikzcd}
V\arrow[tail]{r}{} \arrow[equals]{d}& L_{r} \arrow[two heads]{r}{}\arrow{d} &Y \arrow{d}{r}\\
V\arrow[tail]{r}{} & W\arrow[two heads]{r}{w} & A
\end{tikzcd}
\] 
where $W\in\CW$. We also have that $T\deff L_{r}\in\CT$ by the dual of 
Proposition \ref{prop:f-epic-in-heart-iff-C-f-in-U}, 
as $\ol{r}$ is a monomorphism in $\heart$. 
Thus, by Proposition \ref{prop:variation-of-LN-prop-1-20-dualised-same-start-terms}, we have a conflation 
\[
\begin{tikzcd}[column sep = 1.2cm, ampersand replacement=\&]
T\arrow[tail]{r}{} \&Y\oplus W  \arrow[two heads]{r}{\begin{psmallmatrix}r & w\end{psmallmatrix}} \&A.
\end{tikzcd}
\]

We claim that the morphism $fd\colon K_{0}\to A$ factors through the morphism 
$(\, r \;\, w \,)$. 
Note that the canonical inclusion $\ol{\iota_{Y}}\colon Y\into Y\oplus W$ and canonical projection $\ol{\pi_{Y}}\colon Y\oplus W\onto Y$ are mutually inverse isomorphisms in $\heart=\CB/[\CW]$ as $W\in\CW$. 
Therefore, $\ol{(\, r \;\, w \,)}=\ol{r\pi_{Y}}$ is an epimorphism, and by 
Definition \ref{def:nearly-cokernel-of-morphism-in-heart} 
we have a commutative diagram 
\[
\begin{tikzcd}[ampersand replacement=\&]
Y\oplus W \arrow[tail]{r}{} \arrow{d}[swap]{\begin{psmallmatrix} r & w \end{psmallmatrix}}\& W' \arrow[two heads]{r}{} \arrow{d}\& S\arrow[equals]{d} \\
A \arrow[tail]{r}{} \& U\arrow[two heads]{r}{} \& S
\end{tikzcd}
\] 
in which $U\deff C_{\begin{psmallmatrix} r & w \end{psmallmatrix}}\in\CU$ because 
$\ol{(\, r \;\, w \,)}$ is an epimorphism in $\heart$. So, we have a conflation 
\[
\begin{tikzcd}[column sep=1.5cm, ampersand replacement=\&]
Y\oplus W\arrow[tail]{r}{\begin{psmallmatrix} r & w \\ e & g \end{psmallmatrix}} 
\& A\oplus W' \arrow[two heads]{r}{\begin{psmallmatrix} h & i \end{psmallmatrix}} 
\& U
\end{tikzcd}
\] 
by Proposition \ref{prop:variation-of-LN-prop-1-20-same-end-terms}. 
As 
$\begin{tikzcd}[column sep=0.7cm]
K_{0}\arrow[tail]{r}{c} & P_{1}\arrow[two heads]{r}& S_{0}
\end{tikzcd}$ 
is a conflation, we have an exact sequence 
\[
\begin{tikzcd}[column sep=1.3cm]
\CB(P_{1},U) \arrow{r}{\CB(c,U)}& \CB(K_{0},U)\arrow{r} & \BE(S_{0},U)=0
\end{tikzcd}\] 
by Proposition \ref{prop:exact-sequences-from-E-triangle}, where $\BE(S_{0},U)=0$ as $(\CS,\CT)$ is a cotorsion pair and $\CT = \CU$. 
Thus, there exists $j\colon P_{1}\to U$ such that 
$jc=(\, h \;\, i \,) \begin{psmallmatrix} fd \\0 \end{psmallmatrix}$. 
Since $P_{1}$ is projective and $(\, h \;\, i \,)$ is a deflation, there exists $\begin{psmallmatrix}k \\ l\end{psmallmatrix}\colon P_{1}\to A\oplus W'$ such that 
$(\, h \;\, i \,)\begin{psmallmatrix} k\\ l \end{psmallmatrix}=j$. 
In addition, as $(\, r \;\, w \,)$ is a deflation and we have a morphism $k\colon P_{1}\to A$, there exists $\begin{psmallmatrix} m\\n \end{psmallmatrix}\colon P_{1}\to Y\oplus W$ such that $(\, r \;\, w \,)\begin{psmallmatrix} m\\ n \end{psmallmatrix}=k$. 
Notice that 
$(\, h \;\, i \,)\left[
\begin{psmallmatrix} fd\\ 0  \end{psmallmatrix} - \begin{psmallmatrix} k\\ l \end{psmallmatrix}c\right] = jc-jc=0$. 
Hence, there exists $\begin{psmallmatrix}p\\ q\end{psmallmatrix}\colon K_{0}\to Y\oplus W$ such that $\begin{psmallmatrix} r & w\\ e & g \end{psmallmatrix}\begin{psmallmatrix} p\\ q \end{psmallmatrix}=\begin{psmallmatrix} fd\\ 0  \end{psmallmatrix} - \begin{psmallmatrix} k\\ l \end{psmallmatrix}c.$ 
In particular, we see that 
\[
fd 
= (\, r \;\, w \,)\begin{psmallmatrix} p\\q \end{psmallmatrix}+kc 
= (\, r \;\, w \,)\begin{psmallmatrix} p\\q \end{psmallmatrix} + (\, r \;\, w \,)\begin{psmallmatrix} m\\n \end{psmallmatrix}c
= (\, r \;\, w \,) \begin{psmallmatrix} p+mc \\ q+nc \end{psmallmatrix},
\] 
where $\begin{psmallmatrix} p+mc \\ q+nc \end{psmallmatrix}\colon K_{0}\to Y\oplus W$.

Therefore, by (ET$3^{\op}$), we get a commutative diagram
\begin{equation}\label{eqn:morphism-z-p-q-f}
    \begin{tikzcd}[column sep=1.5cm, row sep=1.2cm, ampersand replacement=\&]
        K_{1}\arrow[tail]{r}{b} \arrow[dotted]{d}{s}\& K_{0}\arrow[two heads]{r}{d}\arrow{d}{\begin{psmallmatrix} p+mc \\ q+nc \end{psmallmatrix}} \& X\arrow{d}{f}\\
        T\arrow[tail]{r}{} \& Y\oplus W \arrow[two heads]{r}[swap]{\begin{psmallmatrix} r&w \end{psmallmatrix}} \& A
    \end{tikzcd}
\end{equation}
Applying Proposition \ref{prop:exact-sequences-from-E-triangle} to conflation \eqref{eqn:conflation-K1-P1-S1}, there is an exact sequence 
\[
\begin{tikzcd}[column sep=1.3cm]
\CB(P_{1},T)\arrow{r}{\CB(a,T)}&\CB(K_{1},T)\arrow{r}{}&\BE(S_{1},T)=0,
\end{tikzcd}
\] 
where $\BE(S_{1},T)=0$ as $(\CS,\CT)$ is a cotorsion pair. 
So, there exists $t\colon P_{1}\to T$ such that $ta=s$. 
This implies $s=ta=tcb$, so by \cite[Cor.\ 3.5]{NakaokaPalu-extriangulated-categories-hovey-twin-cotorsion-pairs-and-model-structures} $f$ must factor through 
$(\, r \;\, w \,)$. 
Thus, there is a morphism 
$\begin{psmallmatrix} u\\ v \end{psmallmatrix}\colon X\to Y\oplus W$ such that 
$(\, r \;\, w \,)\begin{psmallmatrix}u\\ v  \end{psmallmatrix}=f$. 

In $\heart$ we then have $\ol{f}=\ol{(\, r \;\, w \,)}\ol{\begin{psmallmatrix}u\\ v  \end{psmallmatrix}}=\ol{ru}+\ol{wv}=\ol{ru}$ because, for example, $\ol{w}=0$ since $W\in\CW$. 
Hence, in $\heart_{\CR}$ we have that 
$[\ol{f},\ol{1_{A}}]_{\LF}=[\ol{ru},\ol{1_{A}}]_{\LF}=[\ol{r},\ol{1_{A}}]_{\LF}\circ [\ol{u},\ol{1_{Y}}]_{\LF}=L_{\CR}(\ol{r})L_{\CR}(\ol{u})$, which implies 
\[
[\ol{f},\ol{r}]_{\LF}=L_{\CR}(\ol{r})^{-1}\circ [\ol{f},\ol{1_{A}}]_{\LF}= L_{\CR}(\ol{u})=L_{\CR}F(u+[\CS](X,Y))=G(u+[\CS](X,Y)).
\] 
Thus, $G\colon \ol{\SH}\to \heart_{\CR}$ is a full functor.
\end{proof}
Therefore, we have found a fully faithful, dense functor $G\colon \ol{\SH}\to\heart_{\CR}$, which establishes the main result of this section.
\begin{thm}\label{thm:T-is-U-implies-localisation-of-H-at-regular-morphisms-is-equivalent-to-heart-of-cotorsion-pair-S-T}
Let
$(\CB,\BE,\fs)$ be an extriangulated category with enough projectives and injectives. Suppose $((\CS,\CT),(\CU,\CV))$
is a twin cotorsion pair on $\CB$ that satisfies $\CT=\CU$. Let
$\CR$ denote the class of regular morphisms in the heart $\heart$
of $((\CS,\CT),(\CU,\CV))$. Then the Gabriel-Zisman localisation
$\heart_{\CR}$ is equivalent to the heart $\ol{\SH}_{(\CS,\CT)}$
of the single twin cotorsion pair $(\CS,\CT)$.
\end{thm}
%

%


\begin{acknowledgements}
The authors would like to thank Thomas Br{\"u}stle and Robert J.\ Marsh for their support and guidance. 
This work began while the second author was visiting Sherbrooke, and he thanks the algebra group at Universit{\'e} de Sherbrooke for their hospitality and financial help. 
He also gratefully acknowledges financial support from the London Mathematical Society for the visit.

The authors are grateful to Dixy Msapato for spotting a typo in a previous version.

The first author is supported by Bishop's University, Universit{\'e} de Sherbrooke and NSERC of Canada. 
The second author is grateful for financial support from the EPSRC grant EP/P016014/1 ``Higher Dimensional Homological Algebra''. 
Some of this work was carried out during the second author's Ph.D.\ at the University of Leeds.
\end{acknowledgements}
\bibliographystyle{mybst}
\bibliography{references}
\end{document}